\documentclass[12pt]{amsart}
\usepackage{geometry}
\usepackage{amsmath,amssymb}
\usepackage{amsfonts}
\usepackage{eucal}
\usepackage{amsthm}
\usepackage{mathrsfs}
\usepackage{xcolor}
\numberwithin{equation}{section}
\usepackage{comment}

\usepackage{mathtools}

\newcommand{\jap}[1]{\langle #1 \rangle}

\def\a{\alpha}
\def\b{\beta}
\def\c{\gamma}

\def\e{\varepsilon}
\def\f{\varphi}
\def\g{\psi}

\def\l{\lambda}
\def\m{\mu}
\def\n{\nu}
\def\o{\omega}
\def\s{\sigma}
\def\t{\tau}
\def\x{\xi}
\def\y{\eta}
\def\z{\zeta}

\renewcommand{\O}{\Omega}

\newcommand{\Op}{\mathrm{Op}}
\def\re{\mathbb{R}}

\def\pa{\partial}

\renewcommand{\Re}{\text{{\rm Re}\;}}
\renewcommand{\Im}{\text{{\rm Im}\;}}

\newcommand{\supp}{\text{{\rm supp}\;}}

\newtheorem{thm}{Theorem}[section]
\newtheorem{lem}[thm]{Lemma}
\newtheorem{prop}[thm]{Proposition}
\newtheorem{cor}[thm]{Corollary}

\theoremstyle{definition}
\newtheorem{defn}{Definition}
\newtheorem{ass}{Assumption}

\newtheorem{example}{Example}

\theoremstyle{remark}
\newtheorem{rem}[thm]{Remark}

\title[]%
{Local time decay for fractional Schr\"odinger operators with slowly decaying potentials and a weaker Agmon type estimate in a classically forbidden region}

\author{Kouichi Taira}
\address{Department of Mathematical Sciences, Ritsumeikan University, 1-1-1 NojiHigashi, Kusatsu, 525-8577 Japan}
\email{ktaira@fc.ritsumei.ac.jp}
\date{}

\begin{document}
\maketitle

\begin{abstract}

A local time decay estimate of fractional Schr\"odinger operators with slowly decaying positive potentials are studied. It is shown that its resolvent is smooth near zero and the time propagator has fast local time decay which is very different from very short-range cases.
The key element of the proof is to establish a weaker Agmon estimate for a classically forbidden region using exotic symbol calculus. As a byproduct, we prove that the Riesz operator is a pseudodifferential operator with an exotic symbol. 

\end{abstract}


\section{Introduction}

\subsection{Introduction}

We consider fractional Schr\"odinger operators on $\re^n$ with $n\geq 1$ such as
\begin{align}\label{fracex}
P=(-\Delta)^{m}+V(x),\quad \quad P=(-\Delta+1)^{m}-1+V(x)
\end{align}
or these generalizations, where a real-valued potential $V$ is slowly decreasing and positive repulsive: $V(x)\sim c|x|^{-\m}$ as $|x|\to \infty$ with $c>0$ and $0<\m<2m$ in the former case and $0<\m<2$ in the latter case. See Assumption \ref{potass} for the precise assumption. In this paper, we focus on resolvent estimates and local time decay estimates for fractional Schr\"odinger operators between weighted $L^2$-spaces.

Local time decay estimates of the propagator of a Schr\"odinger operator $-\Delta+V(x)$ with decaying potential $V$ has a long history initiated by Jensen-Kato \cite{JK}. This question is mostly related to low frequency regimes. Particles that have non-zero frequency escape to infinity and this contribution to the time decay is less essential unless you consider variable coefficient cases, where there are non-trivial effects from the Hamiltonian flow. Moreover, analysis of the low-frequency regimes strongly depends on the decay rate of the potential function $V$ at infinity. Namely, if we suppose that the decay of $V$ at infinity is like $\jap{x}^{-\m}=(1+|x|^2)^{-\frac{\m}{2}}$, then large time behavior of the Schr\"odinger propagator $e^{-it(-\Delta+V)}$ dramatically changes with the exponent $\m$:

\begin{itemize}
\item (Very short-range cases: $\m>2$) If $V$ decays fast enough and we assume that $-\Delta+V$ has no eigenvalues and no resonances at the zero energy, then the optimal decay rates of the propagator is $O(\jap{t}^{-\frac{n}{2}})$, that is $\|\chi e^{-it(-\Delta+V)}\chi\|_{L^2\to L^2}=O(\jap{t}^{-\frac{n}{2}})$ as $|t|\to \infty$ for $\chi\in C_c^{\infty}(\re^n)$ (see \cite{JK} for $n=3$, \cite{J1} for $n=4$, \cite{J2} for $n\geq 5$, and \cite{JN} for $n=1,2$). We also note that this decay rate coincides with the decay rate of the integral kernel for $V=0$, where we recall $e^{it\Delta}(x,y)=(4\pi it)^{-\frac{n}{2}}e^{-\frac{|x-y|^2}{4it}}$. For generic potentials, the contribution of quantum particles with low frequency on a region $\{|x|\lesssim \l^{\frac{1}{2}}\}$ is negligible due to the uncertainty principle. Such phenomena can be analyzed by Hardy's inequality or Nash's inequality (see \cite{BB,JK}). Its generalization to a certain class of $p_0(D_x)+V$ with a Morse function $p_0$ is studied in \cite{Mu}.

\item (Slowly decaying cases: $0<\m<2$) In this case, the time decay of the propagator strongly depends on the geometry of the potential function $V$. If $V$ is positive repulsive such as $V(x)=\jap{x}^{-\m}$, then we obtain a stronger time decay $\|\chi e^{-it(-\Delta+V)}\chi\|_{L^2\to L^2}=O(\jap{t}^{-\infty})$ as $|t|\to \infty$ for $\chi\in C_c^{\infty}(\re^n)$ (\cite{N2,Y1}). On the other hand, if  $V$ is negative attractive such as $V(x)=-\jap{x}^{-\m}$, then the possible decay rate is $O(\jap{t}^{-1})$ and it is optimal (\cite{FS,Y1}). Finally, we note that the slowly decaying cases cannot be dealt with a perturbation theory, unlike the very short-range cases. Therefore, the strategy used in \cite{JK} does not work there.

\item (Scaling critical case: $\m=2$) In \cite{W}, the time decay property for $-\Delta+\frac{q(x/|x|)}{|x|^2}$ on $\re^n$ (or more general conic manifolds) with $q\geq -\frac{(n-2)^2}{4}$ is studied. The decay rate for $e^{-it(-\Delta+\frac{q(x/|x|)}{|x|^2})}$ is like $O(\jap{t}^{-\n-1})$, where $\n=\sqrt{\frac{(n-2)^2}{4}+\l_0}$ and $\l_0$ is a smallest eigenvalue of $-\Delta_{\mathbb{S}^{n-1}}+q(\o)$ on the sphere. We observe that if the infimum of $q$ is larger and larger, then the decay rate $\jap{t}^{-\n-1}$ is better and better and getting close to the positive repulsive case for $0<\m<2$.  On the other hand, if we take $q(\o)=-\frac{(n-2)^2}{4}$, then the decay rate becomes $\jap{t}^{-1}$, which is similar to the negative attractive case for $0<\m<2$. 

\end{itemize}

These generalizations to Laplace-Beltrami operators on non-compact manifolds have been also investigated by many people in connection with Price's law in black-holes space-times (\cite{BH,BB,GHS,H}).

In the last couple of years, fractional operators such as \eqref{fracex} have been extensively studied in several contexts. For instance, see the time decay estimates for \cite{FSWY,FSY}, the Kato smoothing estimates and the Strichartz estimates \cite{Hos,MY0,MY}, the dispersive estimates \cite{EGT1,EGT2,GG1,GG2,GT}, the heat kernel estimates for \cite{BGJP,DDY,JW}, and absence of positive eigenvalues \cite{HS}. 

In the fractional cases, the value $\m$ at the borderline between the very short-range cases and the slowly decaying cases should change according to the scaling structure at zeros of the symbol: Such a value should be $\m=2m$ for $P=(-\Delta)^m+V(x)$ and $\m=2$ for $(-\Delta+1)^{m}-1$. 
Feng-Soffer-Wu-Yao \cite{FSWY} and Feng-Soffer-Yao \cite{FSY} studied the time decay estimates for the very short-range cases and showed that $\|\jap{x}^{-s} e^{-itP}\jap{x}^{-s'}\|_{L^2\to L^2}=O(\jap{t}^{-\frac{n}{2m}})$ for $P=(-\Delta)^m+V(x)$ and $s,s'>n/2$ when $V$ has sufficiently decay at infinity under the condition that $P$ has no zero eigenvalues and zero resonances. They also investigated how worse the time decay is if $P$ has zero eigenvalues or zero resonances.
The global Kato smoothing estimates and their application to the Stricharz estimates were studied by Mizutani-Yao in \cite{MY0} for the very short-range cases and in \cite{MY} for the scaling critical case.
However, as far as the author knows, previous researches about the low-frequency regimes mostly concern the very short-range cases or the scaling critical cases. There have been no studies of the fractional operators with slowly decreasing potentials so far.

The purpose of this paper is to initiate an analysis of the long-time behavior of fractional Schr\"odinger operators with slowly decreasing potentials such as $P=(1-\Delta)^m-1+V(x)$ and $P=(-\Delta)^m+V(x)$. Here we focus on the positive repulsive cases like $V(x)=\jap{x}^{-\m}$. 
The negative repulsive cases like $V(x)=-\jap{x}^{-\m}$ is also important, however, the method used in \cite{FS} for the usual Schr\"odinger operator $-\Delta+V(x)$ with negative attractive $V$ may be available even in the fractional case. On the other hand, there are some difficulties in applying the argument in \cite{N2} about the positive repulsive $P=-\Delta+V(x)$ in particular for an Agmon-type estimate in a classically forbidden region similar to \cite[Theorem 1.1]{N2}. See Subsection \ref{subsecKey} for more detail. The original motivation of this paper is to give another proof of it. Namely, the main contribution of this paper is to give a proof of an (weaker but sufficient for our application) Agmon-type estimate in a classically forbidden region for more general operators and give an application to local time decay estimates of these Schr\"odinger propagators.
We stress that the method in \cite{BCD} seems depend on the second-order structure of $-\Delta+V(x)$ and in fact, exponentially small estimates with respect to the energy parameter $\l$ like \cite[Theorem III-1]{BCD} and \cite[Theorem 1.1]{N2} seems not hold for operators such as $P=(-\Delta+1)^m-1+V(x)$ unless $m$ is an even integer since the Agmon-type exponential estimate should concern with analyticity of the principal symbol. On the other hand, the method used in this paper is to employ a pseudodifferential technique with some exotic symbol classes. The Agmon-type estimate follows from a disjoint support property and some mapping properties of the Riesz operator while we obtain fast polynomial decay in $\l$ only. For a more precise statement, see Theorem \ref{Agthm}.

As a byproduct, we prove that the Riesz operator $P^{-1}$ and its powers $P^{-N}$ are pseudodifferential operators with exotic symbols (described in Section \ref{subsectionslvame}) and hence the mapping property of the Riesz operator can be deduced from it, see Theorem \ref{Rieszthm}. This result seems new even for the usual Schr\"odinger operator $P=-\Delta+V(x)$. Moreover, we also point out that a similar result does not hold for the very short-range cases or even the simplest case $P=-\Delta$. The symbol of $(-\Delta)^{-N}$ is a singular function, does not satisfy symbolic estimates and the power is well-defined only for $N<n/2$.
 The analysis of the Riesz operator also plays an important role when applying the Mourre theory in Section \ref{sectionMourre}.
 
Finally, we briefly discuss about positive (embedded) eigenvalues. It is well-known that the usual Schr\"odinger operator $-\Delta+V$ has no positive eigenvalues under a mild condition on $V$ (either of $|V(x)|\lesssim \jap{x}^{-\m-1}$ or $|\pa_x^{\a}V(x)|\lesssim \jap{x}^{-\m-|\a|}$ with $\m>0$ is sufficient, see \cite{FHHH,RS}). On the other hand, Herbst-Skibsted (\cite[Subsection 1.2]{HS}) constructed a compactly supported real-valued potential $V$ such that $(-\Delta)^m+V(x)$ has a positive eigenvalue. Therefore, the problem of the existence of positive eigenvalues for the fractional cases is more severe than that for the usual Schr\"odinger operator. In this paper, we are no longer pursuing this problem (except for a few remarks given in Remark \ref{abevarem}) and assume that the operator has no positive eigenvalues.

\subsection{Main theorem}

\begin{ass}\label{symbolass}
Let $p_0\in C^{\infty}(\re^n;[0,\infty))$ satisfying the following two properties:

\noindent$(i)$ There exists $m>0$ and $k\in \mathbb{Z}_{>0}$ such that
\begin{align*}
&|\pa_{\x}^{\a}p_0(\x)|\leq C_{\a}\jap{\x}^{2m-|\a|},\,\, p_0(\x)\gtrsim \jap{\x}^{2m}\,\, \text{for}\,\, |\x|\geq 1,\quad p_0(\x)\sim |\x|^{2k}\,\, \text{for}\,\, |\x|\leq 1.
\end{align*}

\noindent$(ii)$There exists $c>0$ such that
\begin{align*}
\x\cdot \pa_{\x}p_0(\x)\geq cp_0(\x)\quad \text{for}\quad \x\in \re^n.
\end{align*}
\end{ass}

\begin{rem}
Under these assumptions, the range of $p_0$ is $[0,\infty)$ due to the intermediate value theorem. In particular, the spectrum of its Fourier multiplier $p_0(D_x)$ (as a self-adjoint operator with domain $H^{2m}(\re^n)$) is given by $\s(p_0(D_x))=[0,\infty)$.
\end{rem}

Next, we assume that a potential $V$ satisfies a positive repulsive assumption:

\begin{ass}\label{potass}
Let $0<\m<2k$.
Assume that a real-valued smooth function $V$ on $\re^n$ satisfies

\noindent $(i)$ $|\pa_x^{\a}V(x)|\leq C_{\a}\jap{x}^{-\m-|\a|}$

\noindent $(ii)$ There exists $C>0$ such that $V(x)\geq C_1\jap{x}^{-\m}$.

\noindent $(iii)$ There exist $C>0$ and $R_0>0$ such that $-x\cdot \pa_xV(x)\geq C_2\jap{x}^{-\m}$ for $|x|>R_0$.

\end{ass}

\begin{rem}

$(1)$
A typical example of the potential satisfying such conditions is $V(x)=c\jap{x}^{-\m}$ with $c>0$.

$(2)$ The third condition is only needed to prove the strict Mourre inequality \eqref{Mouinprop} for a rescaled operator.
\end{rem}

Now we set
\begin{align}\label{Pdef}
p(x,\x):=p_0(\x)+V(x),\quad P=p_0(D)+V(x),\quad  \quad \m':=\m/k\in (0,2).
\end{align}
Then, $P$ is a self-adjoint operator with the domain $H^{2m}(\re^n)$ by Kato-Rellich theorem. Since $p_0(\x),V(x)\geq 0$, the spectrum of $P$ is given by $\s(P)=[0,\infty)$ by Weyl's theorem.

\begin{example}

Typical examples we want to consider in this paper are the following:
\begin{itemize}
\item $P=(-\Delta)^m+V(x)$ with $m\in\mathbb{Z}_{\geq 1}$ and $0<\m<2m$,

\item $P=(1-\Delta)^m-1+V(x)$ with $m>0$ and $0<\m<2$;
\end{itemize}
where we assume that a smooth function $V$ on $\re^n$ satisfies the three conditions in Assumption \ref{potass}. The degeneracy index $2k$ at $\x=0$ is equal to $2m$ in the former case and $2$ in the latter case since the symbols $|\x|^{2m}$ and $(1+|\x|^2)^m-1$ have Taylor expansions like $|\x|^{2m}$ and $|\x|^2$ as $|\x|\to 0$ respectively.

\end{example}

Finally, we assume that $P$ has no embedded eigenvalues:

\begin{ass}\label{abevass}
$P$ has no eigenvalues.
\end{ass}

\begin{rem}\label{abevarem}

$(1)$
There are no eigenvalues at both the high-frequency and low-frequency regimes, that is there exist $0<\l_{\mathrm{low}}<\l_{\mathrm{high}}$ such that $P$ has no eigenvalues inside $[0,\l_{\mathrm{low}}]\cup [\l_{\mathrm{high}},\infty)$ due to our method. In fact, the strict Mourre inequality holds both on $(0,\l_{\mathrm{low}}]$ and $[\l_{\mathrm{high}},\infty)$. See \eqref{Mouinprop} and the proof of Theorem \ref{thmhighlap}. The zero eigenvalue is excluded by Lemma \ref{mapFredpow}. Hence Assumption \ref{abevass} just imposes that there are no eigenvalues in the middle region $(\l_{\mathrm{low}},\l_{\mathrm{high}})$.
We also note that Assumptions \ref{potass} $(ii), (iii)$ are not needed to exclude eigenvalues in $[\l_{\mathrm{high}},\infty)$.

$(2)$ If we further assume that $-x\cdot\pa_xV(x)\geq 0$ for all $x\in \re^n$, then Assumption \ref{abevass} is satisfied due to a Virial argument as is proved in Theorem \ref{Virthm}.

\end{rem}

We define the outgoing/incoming resolvent by
\begin{align*}
R_{\pm}(z):=\begin{cases}
(P-z)^{-1}\quad \text{for}  \quad z\in\overline{\mathbb{C}}_{\pm}\setminus [0,\infty)\\
(P-\l\mp i0)^{-1}\quad \text{for}  \quad z=\l\in (0,\infty)
\end{cases},
\end{align*}
where we note that $(P-\l\mp i0)^{-1}$ exists for $\l>0$ due to the usual Mourre theory (see Subseciton \ref{subsechigh}).

\begin{thm}\label{fracmainthm}
Let $P=p_0(D_x)+V(x)$ satisfy Assumptions \ref{symbolass}, \ref{potass} and \ref{abevass}.

\noindent$(i)$ $($Uniform estimates for powers of the resolvent$)$
Let $N\in\mathbb{Z}_{>0}$ and $\c>\max(N-1/2,N\left(\frac{1}{2}+\frac{2k-1}{4k}\m\right) )$. Then
\begin{align*}
\sup_{z\in\overline{\mathbb{C}}_{\pm},\,\,0<|z|\leq 1}\jap{z}^{\frac{N}{2}\left(2-\frac{1}{m}\right)}\|\jap{x}^{-\c}R_{\pm}(z)^N\jap{x}^{-\c}\|<\infty.
\end{align*}
Moreover, the map $\overline{\mathbb{C}}_{\pm}\setminus \{0\}\ni z\mapsto \jap{x}^{-\c}R_{\pm}(z)^N\jap{x}^{-\c}\in B(L^2)$ is H\"older continuous. In particular, $\jap{x}^{-\c}R_{\pm}(0)^N\jap{x}^{-\c}=\lim_{\e\searrow 0}\jap{x}^{-\c}R_{\pm}(\e)^N\jap{x}^{-\c}$ exists in $B(L^2)$. Moreover, $R_{+}^N(0)=R_{-}^N(0)$ holds.

\noindent$(ii)$ $($Local time decay of the propagator$)$ For $s>s'>0$, there exists $a(s,s')>0$ such that
\begin{align*}
\|\jap{x}^{-s}e^{-itP}\jap{x}^{-s}\|\lesssim& \jap{t}^{-s'}\quad \text{when}\quad m>\frac{1}{2},\\
\|\jap{x}^{-s}\jap{P}^{-a(s,s')}e^{-itP}\jap{x}^{-s}\|\lesssim& \jap{t}^{-s'}\quad \text{when}\quad m\leq \frac{1}{2}
\end{align*}
In particular, when $m>\frac{1}{2}$, then $\|\chi e^{-itP}\chi\|=O(\jap{t}^{-\infty})$ as $|t|\to \infty$ for $\chi\in C_c^{\infty}(\re^n)$. 

\end{thm}

\begin{rem}

$(1)$
The regularity loss $\jap{P}^{-a(s,s')}$ for $m\leq 1/2$ is more or less needed but we do not pursue the optimal value of $a(s,s')$. On the other hand, we can add some positive weight $\jap{P}^{a(s,s')}$ for $m>1/2$.

$(2)$ Here we consider smooth potentials only. A perturbation argument used in \cite[\S 5.2]{FS} may deal with compactly supported singularities.

$(3)$ By mimicking the proof of the part $(ii)$, we can also prove local time decay estimates for the wave/Klein-Gordon propagators. 

\end{rem}

The proof of the part $(i)$ is divided into two parts: the low-frequency estimates (Theorem \ref{genmainthm}) and the high-frequency estimates (Theorem \ref{thmhighlap}). The proof of the part $(ii)$ is given in Subsection \ref{subsectimedecay}.

Throughout of this paper, we assume Assumptions \ref{symbolass}, \ref{potass} and \ref{abevass}.

\subsection{Limiting absorption principle and time decay estimates}

Our method to obtain the local time decay estimates relies on an analysis of the resolvent of our operator $P$. Here we briefly describe the relationship between the limiting absorption principle (the resolvent estimates) and time decay estimates. For a more rigorous argument, see Subsection \ref{subsectimedecay}.

For simplicity, we assume that an operator $P$ is non-negative self-adjoint operator on a Hilbert space $\mathcal{H}$ and its spectrum is purely absolutely continuous. Then the spectrum theorem shows that 
\begin{align*}
e^{-itP}=\int_0^{\infty}e^{-it\l} dE(\l)=\int_0^{\infty}e^{-it\l} E'(\l) d\l,
\end{align*}
where $E(\l)$ is the spectral projection of $P$ and we use Stone's theorem:
\begin{align*}
E'(\l)=\frac{1}{2\pi i}\lim_{\e\searrow 0}\left((P-\l-i\e)^{-1}-(P-\l+i\e)^{-1} \right).
\end{align*}
We further assume that with some weight $W$,  $(i)$ $WE'(\l)W\in C^{\infty}((0,\infty);B(\mathcal{H}))$, $(ii)$ $WE'(\l)W$ is symbolic at infinity in the sense that $\|W\frac{d^j}{d\l^j}E'(\l)W\|\lesssim \l^{k-\frac{j}{2}}$ for $\l\gtrsim 1$ with some fixed $k\in \re$, $(iii-1)$ 
$WE'(\l)W\in C^{\lfloor\frac{n}{2}\rfloor-1}([0,1);B(\mathcal{H}))$ and $\|W\frac{d^j}{d\l^j}E'(\l)W\|\lesssim \l^{\frac{n}{2}-1-j}$. These conditions are guaranteed for $P=-\Delta$ with $W=\chi(x)\in C_c^{\infty}(\re^n)$ (see \cite[Theorem 1.2]{BB}). Then integration by parts (see \cite[Section 5]{BB}) yields the high-frequency contribution of $We^{-itP}W$ is $O(\jap{t}^{-\infty})$ as $|t| \to \infty$ by $(i),(ii)$ and the low-frequency contribution of $We^{-itP}W$ is just $O(\jap{t}^{-\frac{n}{2}})$ by $(iii)$, where the boundary values at $\l=0$ vanish due to the assumptions $\frac{d^j}{d\l^j}E'(0)=0$ for $j\leq \lfloor\frac{n}{2}\rfloor-1$. In the case of $P=-\Delta+V(x)$ with positive repulsive $V$, Nakamura \cite{N2} showed that the assumption $(iii-1)$ can be improved as $(iii-2)$ $WE'(\l)W\in C^{\infty}([0,1);B(\mathcal{H}))$ and $\frac{d^j}{d\l^j}E'(0)=0$ for each $j=1,2,\cdots$, which implies $\|We^{-itP}W\|=O(\jap{t}^{-\infty})$ as $|t|\to \infty$. This is the reason why we obtain an improved time decay there.

Now our task is to prove an analog of $(i)$, $(ii)$, and $(iii-2)$ even for fractional cases with positive repulsive potentials. The conditions $(i),(ii)$ are relatively easy to check due to a variant of the Mourre theory (see Subsection \ref{subsechigh}) and Assumption \ref{abevass}. Therefore, the bulk of this paper is to show $(iii-2)$.

\subsection{Key estimates}\label{subsecKey}

The key estimate to prove $(iii-2)$ is the following weaker Agmon type estimate in a classically forbidden region.

\begin{thm}\label{Agthm}
Let $L\in \re$ and $\g\in C_c^{\infty}(\re)$. Then there exist $c_0>0$ and $0<\l_0\leq 1$ such that
\begin{align*}
\left\|\jap{x}^L\g(P/\l) 1_{\{|x|\leq c\l^{-1/\m}\}}(x)\right\|=O(\l^{\infty})\quad \text{as}\quad \l\to 0
\end{align*}
for $0<c\leq c_0$.

\end{thm}

The proof of Theorem \ref{Agthm} is given in Section \ref{SectionAg}. This theorem can be interpreted intuitively: No classical particles live in the region $\{(x,\x)\in \re^{2n}\mid  (x,\x)\in \supp\g(p/\l) 1_{\{|x|\leq c\l^{-1/\m}\}}\}$, that is $\supp\g(p/\l) 1_{\{|x|\leq c\l^{-1/\m}}=\emptyset$ for sufficiently small $c>0$ and $0<\l\ll 1$.
This explanation seems applicable even in the case of the very short-range cases $\m>2k$, however, it is not due to the uncertainty principle. The point here is that $h:=\l^{\frac{1}{\m}-\frac{1}{2k}}$ plays the role of the Plank constant after an appropriate scaling, see \cite{N2} or Section \ref{subsecresLAP}. Therefore, to analyze the low-frequency behavior of the resolvent, we only need to deal with the region where both $p(x,\x)\sim\l$ and $|x|\gtrsim \l^{-1/\m}$ holds. In this region, a kind of semiclassical non-trapping estimates hold due to the semiclassical Mourre theory and Virial type conditions at infinity (Assumption \ref{symbolass} $(ii)$ and Assumption \ref{potass} $(iii)$). These are the idea of the proof of $(iii-2)$.

Nakamura \cite[Theorem 1.1]{N2} (for $L=0$, see \cite[Lemma 2.2]{N2} for general $L\in \re$) showed an exponential bound $\left\|\jap{x}^L\g(P/\l) 1_{\{|x|\leq c\l^{-1/\m}\}}  (x)\right\|=O(e^{-c\l^{-(1/\m-1/2)}})$ for $P=-\Delta+V(x)$ with positive repulsive $V$ essentially due to the work in \cite{BCD}. While our result is much weaker than Nakamura's, it covers many other fractional operators and is sufficient to obtain the time decay estimates given in Theorem \ref{fracmainthm} $(ii)$.

We also remark that our operator is deformed to be a semiclassical operator $P_h$ as in \eqref{P_hdef} after a scaling. It is natural to consider that Theorem \ref{Agthm} can be proved by the disjoint support property of $P_h$ like \cite[Theorem 4.25]{Z}. However, the potential part $V_h$ is very singular with respect to the semiclassical parameter $h>0$. Moreover, the principal part of $P_h$ does not have a nice expression if $p_0$ is not a homogeneous function such as $p_0(\x)=(|\x|^2+1)^m-1$. Due to these two reasons, the usual semiclassical method in \cite{Z} cannot be applied at least directly. In this paper, we introduce an exotic symbol class defined by a slowly varying metric \eqref{Riemetdef} and study the pseudodifferential properties of $P$ (not $P_h$) in this symbol class.

\subsection{Related problems}

We have some problems naturally arising from the work in this paper.

\begin{itemize}

\item Prove some local time decay estimates for non-smooth symbols such as $p_0(\x)=|\x|^{2m}$ with $m\in (0,\infty)\setminus \mathbb{Z}_{\geq 1}$. Since our method here heavily depends on the pseudodifferential technique (albeit might be available if $m$ is sufficiently large), this problem is not straightforward at all. Moreover, estimates in a classically forbidden region like Theorem \ref{Agthm} should be weakened due to the singularity of the symbol $|\x|^{2m}$ at $\x=0$.

\item Prove the Strichartz estimates for fractional operators with positive repulsive potentials. The result for the usual Schr\"odinger operator $-\Delta+V(x)$ is given in \cite{M}, see also \cite{T} for the result with a singular potential in dimension two. Very recently, the dispersive estimate for radially symmetric initial data was proved \cite{BTVZ} when the dimension is three.

\item Prove time decay estimates for negative attractive potentials. As is mentioned before, the method for $-\Delta+V(x)$ in \cite{FS} may be available even in the fractional cases. 

\item Study variable coefficient cases such as $-\Delta_g+V(x)$ with a Riemannain metric $g$. As far as the authors know, even second order cases have not been studied so far.

\item Refinement of Theorem \ref{Agthm} assuming analyticity of $p_0$ like \cite[Theorem 1.1]{N2}:. If you assume that $p_0$ is analytic (or more strongly $p_0$ is a polynomial), can you obtain a similar exponential (Agmon) estimate? Moreover, what about its optimality? Perhaps FBI transform method (\cite{Ma,N0}) or commutator arguments (\cite[\S 7.1,\S 7.2]{Z} or \cite{HS}) (with the symbolic estimates used in this paper) are available. The author also anticipates that Theorem \ref{Agthm} is almost optimal if we only assume the smoothness assumption of $p_0$ instead of its analyticity.

\item Equivalence of Sobolev spaces: Which values of $q\in [1,\infty]$ and $s\in \re$ does the equivalence $\|p_0(D_x)^{s}\|_{L^q(\re^n)}\sim \|P^{s}\|_{L^q(\re^n)}$ hold? Such equivalence plays an important role in nonlinear theory. Dinh (\cite[Proposition 2.13]{D}) studied it for $p_0(D_x)=-\Delta$ by using the Gaussian upper bound of the heat kernel. In our case, $p_0(D_x)^{s}P^{-s}$ for $s\in \mathbb{Z}_{\geq 0}$ belongs to the symbol class $S(1,g)$, where $g$ is introduced in Subsection \ref{subsectionslvame}. Therefore, to find $q$ such that the bound $\|p_0(D_x)^{s}\|_{L^q(\re^n)}\lesssim \|P^{s}\|_{L^q(\re^n)}$ holds, we only need to study $L^q$-boundedness of pseudodifferential operators with the symbol class $S(1,g)$.
 See \cite{FMS,Me,KMVZZ} for scaling critical cases.

\end{itemize}







\subsection{Notation}

For an integer $a$, we write $\mathbb{Z}_{\geq a}=\{a,a+1,a+2,\hdots\}$, $\mathbb{C}_{\pm}=\{z\in\mathbb{C}\mid \pm \Im z> 0\}$.
We write $D_{x_j}=i^{-1}\pa_{x_j}$. $\mathcal{S}(\re^n)$ denotes the set of all rapidly decreasing functions on $\re^n$ and $\mathcal{S}'(\re^n)$ denotes the set of all tempered distributions on $\re^n$. We use the weighted Sobolev space: $L^{2,\ell}=\jap{x}^{-\ell}L^2(\re^n)$ and $H^{s,t}=\jap{x}^{-t}\jap{D}^{-s}L^2(\re^n)$ for $s,t\in \re$. For Banach spaces $X,Y$, $B(X,Y)$ denotes the set of all linear bounded operators from $X$ to $Y$. For a Banach space $X$, we denote the norm of $X$ by $\|\cdot\|_{X}$. If $X$ is a Hilbert space, we write the inner metric of $X$ by $(\cdot, \cdot)_{X}$, where $(\cdot, \cdot)_{X}$ is linear with respect to the right variable. We also denote $\|\cdot\|_{L^2}=\|\cdot\|$ and $(\cdot,\cdot)=(\cdot, \cdot)_{L^2}$. The Japanese bracket $\jap{x}$ stands for $(1+|x|^2)^{\frac{1}{2}}$. $1_{\O}$ denotes the characteristic function of a subset $\O$. We write $\mathrm{ad}_TS=[T,S]=TS-ST$ for linear operators $T,S$. For a positive constant $A,B$, $A\lesssim B$ means that there exists a non-essential constant $C>0$ such that $A\leq CB$ and $A\sim B$ means that both $A\lesssim B$ and $B\lesssim A$ hold.

\subsection*{Acknowledgment}
The author was supported by JSPS KAKENHI Grant Number 23K13004.


\section{Preliminary}

\subsection{Scattering symbol class}\label{subsecscsym}

For $a\in \mathcal{S}'(\re^n)$, we define the Weyl quantization $\Op(a)$ of $a$ by
\begin{align*}
\Op(a)u(x):=\frac{1}{(2\pi )^n}\int_{\re^n}\int_{\re^n}e^{i(x-y)\cdot \x}a\left(\frac{x+y}{2},\x\right)u(y)dyd\x.
\end{align*}
For a subset $S\subset\mathcal{S}'(\re^n)$, we write $\Op S:=\{\Op(a)\mid a\in S\}$. See \cite[\S18.5]{Ho} or \cite[\S 4]{Z}.

Let $k,\ell\in \re$. We define scattering symbol classes $S^{k,\ell}$ as follows: For $a\in C^{\infty}(\re^{2n})$,
\begin{align*}
a\in S^{k,\ell} \Leftrightarrow \sup_{(x,\x)\in \re^{2n}}\jap{x}^{-\ell+|\a|}\jap{\x}^{-k+|\b|}|\pa_x^{\a}\pa_{\x}^{\b}a(x,\x)|<\infty\quad \text{for all}\,\, \a,\b\in\mathbb{Z}_{\geq 0}^n.
\end{align*}
If we define a Riemmanan metric $g_{\mathrm{sc}}$ on $\re^{2n}$ by $g_{\mathrm{sc}}=\jap{x}^{-2}dx^2+\jap{\x}^{-2}d\x^2$, then $S^{k,\ell}=S(\jap{x}^{\ell}\jap{\x}^k,g_{\mathrm{sc}})$ holds in the sense of \cite[Definition 18.4.2]{Ho}. Moreover, we define the residual symbol class $S^{-\infty,-\infty}:=\cap_{k,\ell\in \re}S^{k,\ell}$. For $a_j\in S^{k_j,\ell_j}$, there exists $b\in S^{k_1+k_2,\ell_1+\ell_2}$ such that $\Op(b)=\Op(a_1)\Op(a_2)$. We write it as $b=a_1\# a_2$. The symbol of these commutator has a better decay both in $x$ and $\x$ than that of these composition, that is $[\Op(a_1),\Op(a_2)]\in \Op S^{k_1+k_2-1,\ell_1+\ell_2-1}$.
For $s,t\in \re$ and $a\in S^{k,\ell}$, we have $\Op(a)\in B(H^{s+k,t+\ell},H^{s,t})$. For these proof, see \cite[\S18.5]{Ho}. Actually, similar proof to the usual Kohn-Nirenberg symbol class $S^k$ still works.

\subsection{A slowly varying metric}\label{subsectionslvame}

To capture the low-frequency behavior of $P$, we need a more elaborate symbol class. The notation here partially follows the ones in \cite[\S 4.2]{FS} and \cite[\S18.4, \S18.5]{H}.
We define a Riemannian metric $g$ on $\re^{2n}=\re^n_x\times \re^n_{\x}$ by
\begin{align}\label{Riemetdef}
g_{(x,\x)}=\frac{dx^2}{\jap{x}^2}+\jap{x}^{\m'}\frac{d\x^2}{\jap{\x}^2},
\end{align}
where we recall $\m'=\m/k\in (0,2)$ defined in \eqref{Pdef}. We consider a $2n\times 2n$-matrix $A$ defined by
\begin{align*}
A=\begin{pmatrix}
0&-I\\
I&0
\end{pmatrix},
\end{align*}
where $I$ is the $n\times n$ identity matrix.

\begin{prop}\label{propslowvary}

\noindent$(i)$ $g$ is slowly varying, that is there exists $c>0$ and $C\geq 1$ such that $g_{\rho}(\rho')\leq c$ implies
\begin{align*}
C^{-1}g_{\rho}(\t')\leq g_{\rho+\rho'}(\t')\leq Cg_{\rho}(\t'),\quad \t'\in \re^{2n}.
\end{align*}

\noindent$(ii)$ $g$ is $A$-temperate in the sense that there exists $C,N>0$ such that
\begin{align*}
g_{\rho'}(\t)\leq C(1+g_{\rho'}^A(\rho-\rho'))^Ng_{\rho}(\t)
\end{align*}\label{gAdef}
for all $\rho,\rho',\t\in \re^{2n}$, where we define
\begin{align}
g^A_{\rho'}(\rho):=\sup\{|(\rho,\t)|^2\mid g_{\rho'}(A\rho)<1 \}=\sup_{\t\in \re^{2n}\setminus \{0\}}\frac{|(\rho,\t)|^2}{g_{\rho'}(A\t)}.
\end{align}

\noindent$(iii)$ We define $\s_{g}$ by
\begin{align}\label{sigmagdef}
\s_{g}(\rho)^2:=\sup_{\rho'\in \re^{2n}\setminus \{0\}}\frac{g_{\rho}(\rho')}{g_{\rho}^A(\rho')}.
\end{align}
Then we have $\s_g(x,\x)=\jap{x}^{\frac{\m'}{2}-1}\jap{\x}^{-1}$ and hence $\s_g(x,\x)\leq 1$.

\end{prop}

The proof of this proposition is given in Appendix \ref{appendixslvamet} although its proof is more or less standard.

Next, we introduce weight functions that are uniformly bounded with respect to some parameters.

\begin{defn}
Let $\O$ be a set. 
A family of positive function $\{m_{\o}\}_{\o\in \O}$ on $\re^{2n}$ is called a uniformly tempered weight if
\begin{itemize}
\item (Uniform $g$-continuity) There exist $c>0$ and $C\geq 1$  independent of $\o\in \O$ such that
\begin{align*}
g_{\rho}(\rho')\leq c\Rightarrow C^{-1}m_{\o}(\rho)\leq m_{\o}(\rho+\rho')\leq Cm_{\o}(\rho').
\end{align*}

\item (Uniform $(A,g)$-temperateness) There exist $C,N>0$ independent of $\o\in \O$ such that
\begin{align*}
m_{\o}(\rho')\leq C(1+g_{\rho'}^A(\rho-\rho'))^Nm_{\o}(\rho)
\end{align*}
for all $\rho,\rho'\in \re^{2n}$.

\end{itemize}
If $m_{\o}$ is independent of $\o\in \O$, we just call it a tempered weight.

\end{defn}

\begin{rem}\label{remtemwei}

$(1)$ If $\{m_{\o}\}_{\o\in \O}$ is a uniformly tempered weight, then $\{m_{\o}^{-1}\}_{\o\in \O}$ a uniformly tempered weight. If $\{m_{1,\o}\}_{\o\in \O},\{m_{2,\o}\}_{\o\in \O}$ are uniformly tempered weights, then $\{m_{1,\o}+m_{2,\o}\}_{\o\in \O}, \{m_{1,\o}m_{2,\o}\}_{\o\in \O}$ are also uniformly tempered weights.

$(2)$ $\s_{g}$ is a tempered weight.

\end{rem}

For families of symbols $\{a_{\o}\}_{\o\in\O}\subset C^{\infty}(\re^{2n})$ and a uniformly tempered weight $\{m_{\o}\}_{\o\in \O}$, we call $\{a_{\o}\}_{\o\in \O}\in S_{\O,\mathrm{unif}}(\{m_{\o}\}_{\o\in \O},g)$ or simply $a_{\o}\in S_{\O,\mathrm{unif}}(m_{\o},g)$ if for all $\a,\b\in \mathbb{N}^n$, there exist constants $C_{\a\b}$ such that
\begin{align*}
|\pa_{x}^{\a}\pa_{\x}^{\b}a_{\o}(x,\x)|\leq C_{\a\b}m_{\o}(x,\x)\jap{x}^{\frac{\m'}{2}|\b|-|\a|}\jap{\x}^{-|\b|}\quad \text{for all}\quad (x,\x,\o)\in \re^{2n}\times \O.
\end{align*}
Moreover, we write
\begin{align*}
S_{\O,\mathrm{unif}}^{k,\ell}(g):=S_{\O,\mathrm{unif}}(\jap{x}^{\ell}\jap{\x}^k,g),\quad S_{\O,\mathrm{unif}}^{-\infty,-\infty}:=\bigcap_{k,\ell\in \re}S_{\O,\mathrm{unif}}^{k,\ell}(g)
\end{align*}
for $k,\ell \in \re$. We note that $a\in S_{\O,\mathrm{unif}}^{-\infty,-\infty}$ if and only if $\sup_{(x,\x)\in \re^{2n},\o\in \O}\jap{x}^{M}\jap{\x}^{M}|\pa_x^{\a}\pa_{\x}^{\b}a_{\o}(x,\x)|<\infty$\ for all $\a,\b\in\mathbb{Z}_{\geq 0}^n$ and $M>0$.

Next, we state some fundamental properties of the symbol classes $S_{\O,\mathrm{unif}}^{k,\ell}(g)$ and its quantization.

\begin{prop}\label{unifsymprop}

Let $\{m_{\o}\}_{\o\in \O}$ be a uniformly tempered weight and $a_{j,\o}\in S_{\O,\mathrm{unif}}(m_j,g)$ for $j=1,2$. Let $k,l\in \re$.

\noindent$(i)$ $($Composition formula$)$
 Then there exists $b_{\o}\in S_{\O,\mathrm{unif}}(m_{1,\o}m_{2,\o},g)$ such that $b_{\o}-a_{1,\o}a_{2,\o}\in S_{\O,\mathrm{unif}}(m_{1,\o}m_{2,\o}\s_g,g)$ and
\begin{align*}
\Op(b_{\o})=\Op(a_{1,\o})\Op(a_{2,\o}).
\end{align*}
We denote $b_{\o}=a_{1,\o}\# a_{2,\o}$. For each $J\in\mathbb{Z}_{\geq 0}$, we also have
\begin{align}\label{compasyexp}
b_{\o}(x,\x)-\sum_{j=0}^J\frac{i^j}{j!}A(D)(a_1(x,\x)a_2(y,\y))|_{x=y,\x=\y}\in S_{\O,\mathrm{unif}}(m_{1,\o}m_{2,\o}\s_g^{J+1},g),
\end{align}
where $A(D):=\frac{1}{2}(D_{\x}\cdot D_{y}-D_x\cdot D_{\y})$.
In addition, there exists $\tilde{b}_{\o}\in S_{\O,\mathrm{unif}}(m_{1,\o}m_{2,\o},g)$ such that $a_{1,\o}\# a_{2,\o}-\tilde{b}_{\o}\in S_{\O,\mathrm{unif}}^{-\infty,-\infty}$ and 
\begin{align*}
\supp \tilde{b}_{\o}\subset \supp a_{1,\o}\cup \supp a_{2,\o}
\end{align*}

\noindent$(ii)$ $($Commutator$)$
$[\Op(a_{1,\o}), i\Op(a_{2,\o})]\in \Op S_{\O,\mathrm{unif}}(m_{1,\o}m_{2,\o}\s_g,g)$, where we recall that $\s_g$ is defined in \eqref{sigmagdef}.

\noindent$(iii)$ $($Borel's theorem$)$ Let $b_{\o,j}\in  S_{\O,\mathrm{unif}}(m_{\o}\s_g^{j},g)$ for $j=0,1,2,\cdots$. Then there exists $b_{\o}\in S_{\O,\mathrm{unif}}(m_{\o},g)$ such that $\supp b_{\o}\subset \cup_{j\in\mathbb{Z}_{\geq 0}}\supp b_{\o,j}$ and $b_{\o}\approx \sum_{j=0}^{\infty}b_{\o,j}$ in the sense that
\begin{align*}
b_{\o}- \sum_{j=0}^{N}b_{\o,j}\in S_{\O,\mathrm{unif}}(m_{\o}\s_g^{N+1},g)
\end{align*}
for all $N\in\mathbb{Z}_{\geq 1}$.

\noindent$(iv)$ $($Mapping properties$)$ Let $a_{\o}\in S_{\O,\mathrm{unif}}^{k,\ell}(g)$. Then, for all $s,t\in \re$, we have $\Op(a_{\o})\in B(H^{s+k,t+\ell}, H^{s,t})$ and $\sup_{\o\in \O}\|\Op(a_{\o})\|_{B(H^{s+k,t+\ell}, H^{s,t})}<\infty$.

\noindent$(v)$ $($Compactness$)$ Suppose $k,\ell<0$ and $a_{\o}\in S_{\O,\mathrm{unif}}^{k,\ell}(g)$. Then $\Op(a_{\o})$ is a compact operator in $B(H^{s,t})$ for all $s,t\in \re$.

\end{prop}

For its proof, see \cite[\S18.5,\S18.6]{Ho} although the proof is more or less standard.

\begin{lem}\label{lemweight}
\noindent$(i)$ $\{p+\l\}_{\l\in [0,1]}$ a uniformly tempered weight.

\noindent$(ii)$
Let $\a\in\mathbb{Z}_{\geq 0}^n$. Then $\x^{\a}\in S_{[0,1],\mathrm{unif}}(\jap{\x}^{\frac{\max(k-m,0)}{k}|\a|} (p+\l)^{\frac{|\a|}{2k}},g)$.

\end{lem}

We leave the proof of this lemma to Appendix \ref{appendixslvamet}. The proof of $(i)$ is more cumbersome than it looks. The part $(ii)$ is important since it implies that the information of the zero of $\x^{\a}$ is encoded in this symbol class.

\subsection{Functional calculus}
This subsection collects some useful results from the functional analysis and the basis theory of microlocal analysis.
For a Borel measurable function $f$ on $\re$, the operator $f(P)$ of $P$ on $L^2$ is defined by the functional calculus. The results in this subsection hold without Assumption \ref{potass} $(ii),(iii)$.

Now we recall the Helffer-Sj\"ostrand formula (\cite[Theorem 14.8]{Z}): For $\g\in C_c^{\infty}(\re)$ or more generally $\g\in C^{\infty}(\re)$ satisfying $|\g^{(j)}(r)|\lesssim \jap{r}^{s-j}$ for some $s$, we denote its almost analytic extension by $\tilde{\g}\in C^{\infty}(\mathbb{C})$. Namely, $\tilde{\g}$ satisfies $\tilde{\g}|_{\re}=\g$, $|\tilde{\g}(z)|\lesssim \jap{z}^s$, $\pa_{\bar{z}}\tilde{\g}(z)=O(\jap{z}^{s-1}\left(\frac{|\Im z|}{\jap{z}}\right)^{\infty})$ and $\supp \tilde{\g}\subset \{z\in\mathbb{C}\mid \Re z\in \supp \g,\,\, |\Im z|\leq \jap{\Re z}\}$ (in particular, we can take $\tilde{\g}\in C_c^{\infty}(\mathbb{C})$ if $\g\in C_c^{\infty}(\re)$). 
Then the Helffer-Sj\"ostrand formula states
\begin{align}\label{HScauchy}
\g^{(j)}(p/\l)=(-1)^jj!\pi^{-1}\int_{\mathbb{C}}\pa_{\bar{z}}\tilde{\g}(z)(p/\l-z)^{-j-1}dm(z)\quad j\in\mathbb{Z}_{\geq 0},
\end{align}
where the cases $j\geq 1$ follow from Green's formula.

\begin{lem}\label{Pfuncpse}
For $s\in \re$, $(P+1)^s,\jap{P}^s\in \Op S^{2ms,0}$. Moreover, for $\g\in C_c^{\infty}(\re)$, we have $\g(P)\in \cap_{s\in \re}\Op S^{s,0}$.
\end{lem}

\begin{proof}
The proof is very standard by using the Helffer-Sj\"ostrand formula and a parametrix construction in the scattering symbol class. An argument similar to \cite[Theorem 14.9]{Z} (or the proof of Theorem \ref{Rieszthm} in this paper) works. We omit the detail.
\end{proof}

\begin{lem}\label{Resxpres}
Let $L\in \re$. 

\noindent$(i)$ We have $\|\jap{x}^L(P-\l z)^{-1}  \jap{x}^{-L}\|\lesssim \l^{-|L|}\jap{\l z}^{\frac{|L|}{2m}} |\Im z|^{-|L|-1}$ for $z\in\mathbb{C}\setminus \re$ and $0<\l\leq 1$.

\noindent$(ii)$ For $\g\in C_c^{\infty}(\re)$, we have $\|\jap{x}^L\g(P/\l)  \jap{x}^{-L}\|\lesssim \l^{-|L|+1}$ for $0<\l\leq 1$.

\end{lem}

\begin{proof}
\noindent$(i)$
We follow the argument in \cite[Lemma 2.11]{M}. 
By Stein's complex interpolation and the duality, we may assume $L\in \mathbb{Z}_{\geq 0}$. We prove it by induction. The case $L=0$ directly follows from the spectral theorem. Suppose that $\|\jap{x}^{L}(P-\l z)^{-1}  \jap{x}^{-L}\|\lesssim \l^{-L}\jap{\l z}^{\frac{L}{2m}} |\Im z|^{-L-1}$ is proved. We have
\begin{align*}
\jap{x}^{L+1}(P-\l z)^{-1}  \jap{x}^{-L-1}=&(P-\l z)^{-1}-(P-\l z)^{-1} [P,\jap{x}^{L+1}](P-\l z)^{-1}\jap{x}^{-L-1}\\
=&(P-\l z)^{-1}-(P-\l z)^{-1} [P,\jap{x}^{L+1}](P-\l z)^{-1} \jap{x}^{-L-1}.
\end{align*}
Due to the spectral theorem, we have 
\begin{align*}
\|(P-\l z)^{-1}\|\lesssim \l^{-1}|\Im z|^{-1},\quad \|(P-\l z)^{-1}(P+1)^{\frac{2m-1}{2m}}\|\lesssim \l^{-1}|\Im z|^{-1}\jap{\l z}^{\frac{1}{2m}}.
\end{align*}
Moreover, $(P+1)^{-\frac{2m-1}{2m}}[P,\jap{x}^{L+1}]\jap{x}^{-L}\in B(L^2)$ since $(P+1)^{\frac{2m-1}{2m}}\in \Op S^{2m-1,0}$, $[P,\jap{x}^{L+1}]\in \Op S^{2m-1,L}$ and $\jap{x}^{-L}\in \Op S^{0,-L}$. Hence,
\begin{align*}
\|\jap{x}^{L+1}(P-\l z)^{-1}  \jap{x}^{-L-1}\|\lesssim& \|(P-\l z)^{-1}\|+\|(P-\l z)^{-1} [P,\jap{x}^{L+1}](P-\l z)^{-1} \jap{x}^{-L-1}\|\\
\lesssim&\l^{-1}|\Im z|^{-1}+\l^{-1}|\Im z|^{-1}\cdot \jap{\l z}^{\frac{1}{2m}}\cdot 1\cdot \l^{-L}\jap{\l z}^{\frac{L}{2m}} |\Im z|^{-L-1}\\
\lesssim&\l^{-L}\jap{\l z}^{\frac{L+1}{2m}} |\Im z|^{-L-2}
\end{align*}
due to the induction hypothesis. This completes the proof.

\noindent$(ii)$ This follows from $(i)$ and the Helffer-Sj\"ostrand formula \eqref{HScauchy}, where we also use $\jap{\l z}\lesssim 1$ for $\Re z\in \supp\g$ and $|\Im z|\lesssim 1$.

\end{proof}

Finally, we prove a bi-ideal property of the uniform residual class $\Op S_{\O,\mathrm{unif}}^{-\infty,-\infty}$.

\begin{lem}\label{bi-idelem}
Let $\O$ be a set. Suppose $R_{1,\o},R_{2,\o}\in \Op S_{\O,\mathrm{unif}}^{-\infty,-\infty}$ and $A_{\o}\in B(\mathcal{S}(\re^n),\mathcal{S}'(\re^n))$ such that $\sup_{\o\in \O}\|A_{\o}\|_{B(H^{s,t},H^{s',t'})}<\infty$ for some $s,s',t,t'\in \re$. Then we have $R_{2,\o}A_{\o}R_{1,\o}\in \Op S_{\O,\mathrm{unif}}^{-\infty,-\infty}$.

\end{lem}

\begin{proof}
By our assumptions and the mapping property of $\Op S_{\O,\mathrm{unif}}^{-\infty,-\infty}$, for each $s_j,t_j\in \re$, we have $\sup_{\o\in \O}\|R_{2,\o}A_{\o}R_{1,\o}\|_{B(H^{s_1,t_1},H^{s_2,t_2})}<\infty$. Set $B_{\o}:=R_{2,\o}A_{\o}R_{1,\o}$ and write $B_{\o}=\Op(b_{\o})$ with some $b_{\o}\in \mathcal{S}'(\re^n)$. Such $b_{\o}$ exists due to the Schwartz kernel theorem and the Fourier transform. Then it suffices to prove that $b_{\o}\in \mathcal{S}(\re^{2n})$ and all the seminorms of $b_{\o}$ are uniformly bounded with respect to $\o\in \O$.

We recall from \cite[Theorem 8.1]{Z} that there exists $M>0$ such that $\|b\|_{L^{\infty}(\re^n)}\lesssim\sum_{|\c|\leq M}\|\Op(\pa^{\c}b)\|_{B(L^2)}$ for all $b\in\mathcal{S}'(\re^n)$ if the right hand side is bounded and from \cite[$(8.1.15)$]{Z} that $\mathrm{ad}_{x_j}\Op(b)=-\Op(D_{\x_j}b)$ and $\mathrm{ad}_{D_{x_j}}\Op(b)=\Op(D_{x_j}b)$. These implies that for each $N_1,N_2\in\mathbb{Z}_{\geq 0}$, there exist $s,s',t,t'\in \re$ such that $\|\jap{x}^{2N_1}\jap{D_x}^{2N_2}b\|_{L^{\infty}(\re^n)}\lesssim\|\Op(b)\|_{B(H^{s,t},H^{s',t'})}$ for all $b\in\mathcal{S}'(\re^n)$. Applying this estimate with $b=b_{\o}$, we conclude that $b_{\o}\in\mathcal{S}(\re^n)$ and all the seminorms of $b_{\o}$ are uniformly bounded with respect to $\o\in \O$.

\end{proof}

\section{Weaker Agmon estimate}\label{SectionAg}

In this section, we prove Theorem \ref{Agthm}.
We divide it into two parts: For fixed $0<\e<1/\m$, we prove
\begin{align*}
\left\|\jap{x}^L\g(P/\l) 1_{\{|x|\in [0,\l^{-\e}]\}  }\right\|=O(\l^{\infty}),\quad \left\|\jap{x}^L\g(P/\l) 1_{\{|x|\in [\l^{-\e},c\l^{-1/\m}]\}  }\right\|=O(\l^{\infty})
\end{align*}
for $0<\l\leq 1$ and sufficiently small $c>0$.

\subsection{Analysis of the Riesz operator}\label{subsecinv}

In this subsection, we prove
\begin{align}\label{Ag1}
\left\|\jap{x}^L\g(P/\l) 1_{\{|x|\in [0,\l^{-\e}]\}  }\right\|=O(\l^{\infty})\quad \text{as}\quad \l\to 0.
\end{align}

First, we show that the Riesz operator $P^{-1}$ and its power are pseudodifferential operator with exotic symbol belonging to $S(p^{-1},g)$ or its generalization to a family $(P+\l)^{-1}$. The mapping properties of  the Riesz operator follows from the general theory of the pseudodifferential operators.
Such mapping properties are also used in Section \ref{sectionMourre}. We note that symbol classes $S_{J,\mathrm{unif}}((p+\l)^N,g)$ and $S_{J,\mathrm{unif}}((p+\l)^{-N},g)$ are well-defined by Lemma \ref{lemweight} and Remark \ref{remtemwei} $(1)$.

\begin{thm}\label{Rieszthm}
Set $J=[0,1]$ and $N\in\mathbb{Z}_{\geq 1}$.

\noindent$(i)$ The power $P^N$ as a linear continuous map from $\mathcal{S}(\re^n)$ to itself is invertible. We denote its inverse by $P^{-N}$.

\noindent$(ii)$ $(P+\l)^{-N}\in \Op S_{J,\mathrm{unif}}((p+\l)^{-N},g)$.

\noindent$(iii)$ For all $s,t\in \re$, the family of the operators $\{(P+\l)^{-N}\}_{\l\in J}$ is uniformly bounded in $B(H^{s,t}(\re^n), H^{s+2mN,t-\m N}(\re^n))$.

\end{thm}

\begin{rem}\label{Riszthmrem}
$(1)$
Due to the open mapping theorem on Fr\'echet space and duality, the mapping $P^{N}$ is homeomorphism both on $\mathcal{S}(\re^n)$ and $\mathcal{S}'(\re^n)$.

$(2)$ In \cite[Theorem 5.3]{Y2}, it is shown that $\jap{x}^{-\a_1}P^{-N}\jap{x}^{-\a_2}$ is bounded for $\a_j\geq 0$ with $\a_1-\a_2\geq N\m$ when $P=-\Delta+V(x)$. This mapping properties also follow from our Theorem \ref{Rieszthm} $(iii)$ although some regularity assumptions on $V$ are relaxed in \cite{Y2}.
\end{rem}

The key of the proof of Theorem \ref{Rieszthm} is the following parametrix construction.

\begin{lem}\label{lemparaconpow}

\noindent$(i)$ $p+\l\in S_{J,\mathrm{unif}}(p+\l,g)$ and $(p+\l)^{-1}\in S_{J,\mathrm{unif}}((p+\l)^{-1},g)$.

\noindent$(ii)$ There exist $a_{\l}\sim \sum_{j=0}^{\infty}a_{\l,j}\in S_{J,\mathrm{unif}}((p+\l)^{-1},g)$ such that 
\begin{align*}
(p+\l)\# a_{\l}-1,\, a_{\l}\# (p+\l)-1\in S_{J,\mathrm{unif}}^{-\infty,-\infty}.
\end{align*}

\end{lem}

\begin{proof}

\noindent$(i)$ Due to Remark \ref{remtemwei} $(1)$, it suffices to prove $(p+\l)\in S_{J,\mathrm{unif}}((p+\l),g)$. Moreover, the general case $\l\in I$ follows from the case $\l=0$ since $(p+\l)^{-1}\leq p^{-1}$ and $\pa_{x}^{\a}\pa_{\x}^{\b}\l=0$ when $(\a,\b)\neq (0,0)$. Therefore, it suffices to prove $|p(x,\x)^{-1}\pa_x^{\a}\pa_{\x}^{\b}p(x,\x)|\lesssim \jap{x}^{\frac{\m'}{2}|\b|-|\a|}\jap{\x}^{-|\b|}$.

The case $(\a,\b)=(0,0)$ is dealt with easily. 
If $\a\neq 0$, then $|\pa_x^{\a}\pa_{\x}^{\b}p(x,\x)|=|\pa_x^{\a}\pa_{\x}^{\b}V(x)|\lesssim \jap{x}^{-\m-|\a|}\jap{\x}^{-|\b|}$ since $\pa_x^{\a}p_0(\x)=0$ and $\pa_{\x}^{\b}V(x)=0$ for $\b\neq 0$. Using the inequality $p(x,\x)^{-1}\lesssim \jap{x}^{\m}$, we have
\begin{align*}
|p(x,\x)^{-1}\pa_{x}^{\a}\pa_{\x}^{\b}p(x,\x)|\lesssim \jap{x}^{\m}\cdot \jap{x}^{-\m-|\a|} \jap{\x}^{-|\b|}=\jap{x}^{-|\a|} \jap{\x}^{-|\b|}\leq \jap{x}^{\frac{\m'}{2}|\b|-|\a|} \jap{\x}^{-|\b|}.
\end{align*}

Next, we consider the case where both $\a=0$ and $\b\neq 0$ hold. We note $\pa_{\x}^{\b}p(x,\x)=\pa_{\x}^{\b}p_0(\x)$.
From Assumption \ref{symbolass}, we have $|\pa_{\x}^{\b}p_0(\x)|\lesssim |\x|^{2k-|\b|}$ for $|\x|\leq 1$ and $|\pa_{\x}^{\b}p_0(\x)|\lesssim \jap{\x}^{2m-|\b|}$ for $|\x|\geq 1$.
If $|\b|\leq 2k$, then
\begin{align*}
|p(x,\x)^{-1}\pa_{\x}^{\b}p(x,\x)|\lesssim& (p_0(\x)+\jap{x}^{-\m})\cdot |\pa_{\x}^{\b}p_0(\x)|\\
\lesssim& \begin{cases}(|\x|^{2k}+\jap{x}^{-\m})^{-1}\cdot |\x|^{2k-|\b|} \quad \text{for}\quad |\x|\leq 1 \\
\jap{\x}^{-|\b|} \quad \text{for}\quad |\x|\geq 1
\end{cases}\\
\lesssim&\jap{x}^{\frac{\m}{2k}|\b|}\jap{\x}^{-|\b|}=\jap{x}^{\frac{\m'}{2}|\b|}\jap{\x}^{-|\b|}.
\end{align*}
where we use the inequality $(|\x|^{2k}+\jap{x}^{-\m})^{-1}\cdot |\x|^{2k-|\b|}\lesssim (|\x|^{2k}+\jap{x}^{-\m})^{-\frac{|\b|}{2k}}$.
On the other hand, if $|\b|\geq 2k+1$, then
\begin{align*}
|p(x,\x)^{-1}\pa_{\x}^{\b}p(x,\x)|\lesssim& \begin{cases}(|\x|^{2k}+\jap{x}^{-\m})^{-1} \quad \text{for}\quad |\x|\leq 1\\
\jap{\x}^{-|\b|}\quad \text{for}\quad |\x|\geq 1
\end{cases}\\
\lesssim &\jap{x}^{\m}\jap{\x}^{-|\b|}\leq \jap{x}^{\frac{\m}{2k}|\b|}\jap{\x}^{-|\b|}=\jap{x}^{\frac{\m'}{2}|\b|}\jap{\x}^{-|\b|},
\end{align*}
where we use $|\pa_{\x}^{\b}p_0(\x)|\lesssim 1$ for $|\x|\leq 1$ in the first line instead. Thus we have proved $p\in S(p,g)$, which completes the proof.

\noindent$(ii)$ This part follows from a usual parametrix construction:
We recall $\s_g(x,\x)=\jap{x}^{\frac{\m'}{2}-1}\jap{\x}^{-1}$.
Setting $a_{\l,0}:=(p+\l)^{-1}\in S_{J,\mathrm{unif}}((p+\l)^{-1},g)$ (due to the part $(i)$), we have $(p+\l)\# a_{\l,0}-1 \in S_{J,\mathrm{unif}}(\s_g,g)$ by the composition formula (Proposition \ref{unifsymprop} $(i)$).
Inductively, we define $r_{\l,j+1}:=(p+\l)\# a_{\l,j}+r_{\l,j}\in S_{J,\mathrm{unif}}(\s_g^{j+1},g)$ and  $a_{\l,j+1}:=-(p+\l)^{-1}r_{\l,j+1} \in  S_{J,\mathrm{unif}}((p+\l)^{-1}\s_g^{j+1},g)$. Then it is easy to see that $(p+\l)\#(\sum_{j=0}^Ma_{\l,j})=1-r_{M+1}$. Due to the Borel summation theorem (Proposition \ref{unifsymprop} $(iii)$), there exists $a_{\l}\approx \sum_{j=0}^{\infty}a_{\l,j}\in S_{J,\mathrm{unif}}((p+\l)^{-1},g)$ such that $(p+\l)\# a_{\l}-1\in S_{J,\mathrm{unif}}^{-\infty,-\infty}$. Similarly, we can find $\tilde{a}_{\l}\in S_{J,\mathrm{unif}}((p+\l)^{-1},g)$ such that $\tilde{a}_{\l}\#(p+\l)-1\in S_{J,\mathrm{unif}}^{-\infty,-\infty}$. Then $\tilde{a}_{\l}=\tilde{a}_{\l}\# 1=\tilde{a}_{\l}\# (p+\l)\# a_{\l}=a_{\l}$ modulo $S_{J,\mathrm{unif}}^{-\infty,-\infty}$, which implies $a_{\l}\#(p+\l)-1\in S_{J,\mathrm{unif}}^{-\infty,-\infty}$. This completes the proof

\end{proof}

\begin{lem}\label{mapFredpow}
Let $\l\in J=[0,1]$.

\noindent$(i)$ Let $u\in \mathcal{S}'(\re^n)$ satisfy $(P+\l)u\in \mathcal{S}(\re^n)$. Then we have $u\in \mathcal{S}(\re^n)$.

\noindent$(ii)$ $P+\l:\mathcal{S}'(\re^n)\to \mathcal{S}'(\re^n)$ is injective.

\noindent$(iii)$ $W:=\{(P+\l)u\in L^{2,\frac{\m}{2}}\mid u\in\mathcal{S}(\re^n)\}$ is dense in $L^{2,\frac{\m}{2}}$.

\noindent$(iv)$ 
We set $D_{\m}(P+\l):=\{u\in H^{2m,-\frac{\m}{2}}\mid (P+\l)u\in L^{2,\frac{\m}{2}}\}$.
Then the map
\begin{align*}
P+\l:D_{\m}(P+\l)\to L^{2,\frac{\m}{2}}
\end{align*}
is invertible.

\noindent$(v)$ $\sup_{\l\in J}\|(P+\l)^{-1}\|_{B(L^{2,\frac{\m}{2}},L^{2,-\frac{\m}{2}})}<\infty$.

\end{lem}


\begin{proof}

\noindent$(i)$
Let $a_{\l}$ be the parametrix of $p+\l$ as in Lemma \ref{lemparaconpow} and set $R_{1,\l}:=\Op(a_{\l})(P+\l)-I \in \Op S_{J,\mathrm{unif}}^{-\infty,-\infty}$. If $u\in\mathcal{S}'(\re^n)$ satisfies $(P+\l)u\in \mathcal{S}(\re^n)$, then $u=\Op(a_{\l})(P+\l)^Nu-R_{1,\l}u\in\mathcal{S}(\re^n)$ since $\Op(a_{\l})(P+\l)u\in\mathcal{S}(\re^n)$ by the assumption and $R_{1,\l}u\in\mathcal{S}(\re^n)$ due to $R_{1,\l}\in \Op S_{J,\mathrm{unif}}^{-\infty,-\infty}$.

\noindent$(ii)$ Suppose that $u\in\mathcal{S}(\re^n)$ satisfies $(P+\l)u=0$. By the part $(i)$, we have $u\in\mathcal{S}(\re^n)$. We recall that $P=p_0(D_x)+V(x)$ and $V(x)\geq 0$ by Assumption \ref{potass}. Since $\l\geq 0$, we have
\begin{align*}
0=(u,(P+\l)u)\geq \frac{1}{(2\pi )^n}\int_{\re^n}p_0(\x)| \hat{u}(\x)|^2d\x.
\end{align*}
where $\hat{u}$ denotes the Fourier transform of $u$.
Due to Assumption \ref{symbolass}, it turns out that $p_0(\x)\geq 0$ and that $\x=0$ is the only zero of $p_0$. Thus we conclude $u=0$.

\noindent$(iii)$ Now suppose that $f\in L^{2,\frac{\m}{2}}$ is orthogonal to $W$ with respect to the inner product in $L^{2,\frac{\m}{2}}(\re^n)$, that is
\begin{align*}
(\jap{x}^{\frac{\m}{2}}f,\jap{x}^{\frac{\m}{2}}(P+\l)u)=0
\end{align*}
for all $u\in\mathcal{S}(\re^n)$. Then we have $(P+\l)(\jap{x}^{\m }f)=0$ in the distributional sense. Then the part $(ii)$ implies $f=0$. This proves that $W$ is dense in $L^{2,\frac{\m}{2}}$.

\noindent$(iv)$ By the part $(ii)$, it suffices to prove that the map is surjective.
 We regard $D_{\m}(P+\l)$ as a Banach space equipped with its graph norm. We set $R_{1,\l}:=\Op(a_{\l})(P+\l)-I \in \Op S_{J,\mathrm{unif}}^{-\infty,-\infty}$ and $R_{2,\l}:=(P+\l)^N\Op(a_{\l})-I \in \Op S_{J,\mathrm{unif}}^{-\infty,-\infty}$, where $a_{\l}$ is as in Lemma \ref{lemparaconpow}. Then $R_{1,\l},R_{2,\l}$ are compact operators from $D_{\m}(P+\l)$ to $L^{2,\frac{\m}{2}}$ essentially due to Proposition \ref{unifsymprop} $(v)$. Moreover, $\Op(a_{\l})$ is a bounded linear operator from $L^{2,\frac{\m}{2}}$ to $D_{\m}(P+\l)$ due to these identities and the mapping property of $\Op(a_{\l})$ (see Proposition \ref{unifsymprop} $(iv)$). Then the map $(P+\l):D_{\m}(P+\l)\to L^{2,\frac{\m}{2}}$ is a Fredholm operator by \cite[Corollary 19.1.9]{Ho} and in particular, its range is closed.
On the other hand, the part $(iii)$ implies that this map has a dense range. Hence the mapping $P+\l$ is surjective. 

\noindent$(v)$ We have $(u,(P+\l)u)\geq (u,Vu)\gtrsim \|\jap{x}^{-\frac{\m}{2}}u\|^2$ and hence $\|\jap{x}^{-\frac{\m}{2}}u\|^2\leq \|\jap{x}^{-\frac{\m}{2}}u\|\|\jap{x}^{\frac{\m}{2}}(P+\l)u\|$, which implies $\sup_{\l\in J}\|(P+\l)^{-1}\|_{B(L^{2,\frac{\m}{2}},L^{2,-\frac{\m}{2}})}<\infty$.

\end{proof}

\begin{proof}[Proof of Theorem \ref{Rieszthm}]

\noindent$(i)$ By Lemma \ref{mapFredpow} $(ii)$ and an induction argument, it turns out that $P^N$ is injective on $\mathcal{S}(\re^n)$.
Thus it suffices to prove that $P^N$ is surjective on $\mathcal{S}(\re^n)$. Let $f\in\mathcal{S}(\re^n)$. Due to Lemma \ref{mapFredpow} $(iv)$, there exists $f_1\in D_{\m}(P)$ such that $f=Pf_1$. By Lemma \ref{mapFredpow} $(i)$, we have $f_1\in \mathcal{S}(\re^n)$. Repeating this argument, we obtain $u\in\mathcal{S}(\re^n)$ such that $f=P^Nu$.

\noindent$(ii)$ The claim for $N\geq 2$ directly follows from the case $N=1$ and the composition formula (Proposition \ref{unifsymprop} $(i)$). So we may assume $N=1$. We set
\begin{align*}
R_{1,\l}:=(P+\l) \Op(a_{\l})-I\in \Op S_{J,\mathrm{unif}}^{-\infty,-\infty},\,\, R_{2,\l}:= \Op(a_{\l})(P+\l)-I\in \Op S_{J,\mathrm{unif}}^{-\infty,-\infty},
\end{align*}
where $a_{\l}\in S_{J,\mathrm{unif}}((p+\l)^{-1},g)$ is as in Lemma \ref{lemparaconpow} $(ii)$. Then we have
\begin{align*}
(P+\l)^{-1}=\Op(a_{\l})-\Op(a_{\l})R_{1,\l}+R_{2,\l}(P+\l)^{-1}R_{1,\l}.
\end{align*}
Since $a_{\l}\in S_{J,\mathrm{unif}}((p+\l)^{-1},g)$ and $R_{1,\l}\in \Op S_{J,\mathrm{unif}}^{-\infty,-\infty}$, we have $\Op(a_{\l}),\Op(a_{\l})R_{1,\l}\in \Op S_{J,\mathrm{unif}}((p+\l)^{-1},g)$ by the composition formula (Proposition \ref{unifsymprop} $(iii)$). On the other hand, we $R_{2,\l}(P+\l)^{-1}R_{1,\l}\in \Op S_{J,\mathrm{unif}}^{-\infty,-\infty}$ due to Lemma \ref{bi-idelem} and Lemma \ref{mapFredpow} $(v)$. Therefore, we conclude $(P+\l)^{-1}\in \Op S_{J,\mathrm{unif}}((p+\l)^{-1},g)$.

The part $(iii)$ follows from the part $(ii)$, the uniform bound $(p+\l)^{-N}\lesssim \jap{x}^{\m N}\jap{\x}^{-2mN}$, and Proposition \ref{unifsymprop} $(iv)$.

\end{proof}

Now we have prepared to prove \eqref{Ag1}.

\begin{proof}[Proof of \eqref{Ag1}] Fix $L\in \re$.
Let $N\in\mathbb{Z}_{>0}$ be large enough and set $M:=N\m+L>0$. 
First, we observe that $\jap{x}^L(P+\l)^{-N}\jap{x}^{-M}$ is uniformly bounded in $B(L^2(\re^n))$ with respect to $\l\in (0,1]$ due to Theorem \ref{Rieszthm} $(iii)$ (with $s=0$ and $t=M$). 
Then, we have $\|\jap{x}^{M}\jap{\l^{\e}x}^{-M}\|\lesssim \l^{-M\e }$ and hence
\begin{align*}
\|\jap{x}^L(P/\l+1)^{-N}\jap{\l^{\e}x}^{-M}\|_{B(L^2(\re^n))}\lesssim \l^{N-\e M}.
\end{align*}
By Lemma \ref{Resxpres} $(ii)$, for $L\in \re$, we have $\|\jap{x}^L\g_K(P/\l)  \jap{x}^{-L}\|\lesssim \l^{-|L|+1}$, where we set $\g_K(s)=(s+1)^K\g(s)$. Then we have
\begin{align*}
\|\jap{x}^L\g(P/\l)\jap{\l^{\e}x}^{-M}\|\leq&\|\jap{x}^L\g_K(P/\l)\jap{x}^{-L}\| \|\jap{x}^L(P/\l+1)^{-N}\jap{\l^{\e}x}^{-M}\|\\
\lesssim&\l^{N+1-\e M-|L|}=\l^{(1-\e\m)N+1-\e L- |L| }.
\end{align*}
Thus we conclude
\begin{align*}
\left\|\jap{x}^L\g(P/\l) 1_{\{|x|\in [0,\l^{-\e}]\}  }\right\|\lesssim\|\jap{x}^L\g(P/\l)\jap{\l^{\e}x}^{-M}\|\lesssim\l^{(1-\e\m)N+1-\e L- |L| }.
\end{align*}
Since $1-\e\m>0$ and since we can take $N\in \mathbb{Z}_{>0}$ arbitrary large, this inequality implies $\left\|\jap{x}^L\g(P/\l) 1_{\{|x|\in [0,\l^{-\e}]\}  }\right\|=O(\l^{\infty})$, which completes the proof.

\end{proof}

\subsection{Disjoint support property}

Next, we deal with a larger region ${\{\l^{-\e}\leq |x|\leq c\l^{-1/\m}\}}$ and prove a weaker Agmon estimate there. Namely, we prove
\begin{align}\label{Ag2}
\|\jap{x}^L\g(P/\l)1_{\{\l^{-\e}\leq |x|\leq c\l^{-1/\m}\}}\|=O(\l^{\infty})\quad \text{as}\quad \l\to 0,
\end{align}
where we recall that $L\in \re$, $\g\in C_c^{\infty}(\re)$, $0<\e<1/\m$, and $c>0$ is sufficiently small. We remark that most of the following argument does not need the cut-off into the region $\{\l^{-\e}\leq |x|\}$. The only part to use it is the remainder estimate \eqref{xlowbduse} of a parametrix constructed below.

First, we construct a smooth cut-off function supported near $\{\l^{-\e}\leq |x|\leq c\l^{-\frac{1}{\m}}\}$.

\begin{lem}\label{cutoffconst}
Fix $c>0$. Then there exists $\chi_{\l}\in C_c^{\infty}(\re^n;[0,1])$ such that
\begin{align*}
&\chi_{\l}(x)=1\quad \text{for}\quad \l^{-\e}\leq |x|\leq c\l^{-\frac{1}{\m}},\\
&|\pa_x^{\a}\chi_{\l}(x)|\lesssim \jap{x}^{-|\a|},\quad \supp \chi_{\l}\subset \left\{x\in \re^n\mid \frac{1}{2}\l^{-\e}\leq |x|\leq 2c\l^{-\frac{1}{\m}} \right\},
\end{align*}
where the first estimates are uniform in $0<\l\leq 1$.

\end{lem}

\begin{proof}

Let $\chi,\tilde{\chi}\in C_c^{\infty}(\re^n)$ such that $\chi(x)=1$ for $|x|\leq 1/2$ and $\chi(x)=0$ for $|x|\geq 1$. We define $\chi_{\l}(x)=\chi((2c)^{-1}\l^{\frac{1}{\m}}x)(1-\chi(\l^{\e}x))$. Then $\chi_{\l}$ satisfies the support properties as above. (Note: $|\pa_{x}^{\a}(\chi(x/R))|\lesssim \jap{x}^{-|\a|}$ uniformly in $R\gtrsim 1$).

\end{proof}

To prove \eqref{Ag2}, we construct a parametrix of $\g(P/\l)\chi_{\l}$ in the symbol class $S(p+\l,g)$. For $Z>0$, we define
\begin{align*}
D_Z:=\{(\l,z)\in (0,1]\times \mathbb{C}\setminus \re\mid |z|\leq Z \}.
\end{align*}




\begin{lem}\label{parapsymest}
Fix $Z>0$.

\noindent$(i)$ $p-\l z\in S_{D_Z,\mathrm{unif}}((p+\l),g)$.

\noindent$(ii)$
If $c_1>0$ is sufficiently small, then
\begin{align}\label{w/zsymest}
|(p-\l z)^{-1}\pa_x^{\a}\pa_{\x}^{\b}(p(x,\x)-\l z)|\leq&C_{\a\b}\jap{x}^{-|\a|+\frac{\m}{2}|\b|}\jap{\x}^{-|\b|}.
\end{align}
uniformly in $(\l,z)\in D_Z$ and $|x|\leq 2c_1\l^{-\frac{1}{\m}}$. In particular, $(p-\l z)^{-1}\chi_{\l}\in S_{D_Z,\mathrm{unif}}((p+\l)^{-1},g)$ if $c>0$ is sufficiently small.

\end{lem}

\begin{rem}
The function $(p-\l z)^{-1}$ itself belongs to $S(p+\l,g)$ for fixed $z$, however, we do not have $(p-\l z)^{-1}\in S_{D_Z,\mathrm{unif}}((p+\l),g)$.
\end{rem}

\begin{proof}

\noindent$(i)$ This directly follows from Lemma \ref{lemparaconpow} $(i)$.

\noindent$(ii)$
We recall $C_1>0$ is a constant such that $V(x)\geq C_1\jap{x}^{-\m}$. Taking $c_1>0$ small enough, we have $|\l z|\leq \frac{C_1}{2}\jap{x}^{-\m}$ for $(\l,z)\in D_Z$ and $|x|\leq 2c_1\l^{-\frac{1}{\m}}$. This implies $|(p(x,\x)-\l z)^{-1}|\leq (p(x,\x)-\frac{C_1}{2}\jap{x}^{-\m})^{-1}\lesssim (p(x,\x)+\l)^{-1}$ uniformly in such $(\l,z,x)$.

The estimates for $(\a,\b)\neq (0,0)$ follows from the estimates Lemma \ref{lemparaconpow} $(i)$ and the formula
\begin{align*}
\pa_{x}^{\a}\pa_{\x}^{\b}(p(x,\x)-\l z )^{-1}=\sum_{\substack{1\leq \ell \leq |\a|+|\b|\\ \a_1+\a_2+\hdots+\a_{\ell}=\a,\\ \b_1+\b_2+\hdots+\b_{\ell}=\b, \\ (\a_j,\b_j)\neq (0,0) }} c_{\a_1,\hdots\b_{\ell}}(p(x,\x)-\l z )^{-1}\prod_{k=1}^{\ell}\frac{\pa_{x}^{\a_{k}}\pa_{\x}^{\b_k}p(x,\x)}{p(x,\x)-\l z}.
\end{align*}
Moreover, Lemma \ref{cutoffconst} and \eqref{w/zsymest} imply $(p-\l z)^{-1}\chi_{\l}\in S_{D_Z,\mathrm{unif}}((p+\l)^{-1},g)$.

\end{proof}

Next, we construct a uniform parametrix of $(P-\l z)^{-1}\chi_{\l}(x)$.

\begin{lem}\label{uncparalarge}
Let $N,Z>0$.
Then there exists $a_{\l,z}\in S_{D_Z,\mathrm{unif}}((p+\l)^{-1},g)$ such that $(p-\l z)\# a_{\l,z}-\chi_{\l}\in S_{D_Z,\mathrm{unif}}^{-N,-N}(g)$ and $a_{\l,z}$ has the form
\begin{align}\label{parainvpow}
a_{\l,z}(x,\x)=\sum_{j=1}^{J_N}b_{\l,j}(x,\x)(p(x,\x)-\l z)^{-j},
\end{align}
where $b_{\l,j}\in C^{\infty}(\re^{2n})$ are independent of $z$ and $\supp b_{\l,z}\subset \supp \chi_{\l}$.

\end{lem}

\begin{rem}
Note: $|(p-\l z)^{-1}|\lesssim (p+\l)^{-1}$ on $\supp a_{\l,z}$.
\end{rem}

\begin{proof}
The proof is based on the standard parametrix construction although we have to be careful for the symbol class and supports of symbols. We briefly explain it here. See \cite[Theorem 14.9]{Z} for a similar construction.
 We write $D=D_Z$ for simplicity. Recall $\s_g(x,\x):=\jap{x}^{\frac{\m'}{2}-1}\jap{\x}^{-1}$ from \eqref{sigmagdef}.

Set $a_{\l,z,0}:=(p-\l z)^{-1}\chi_{\l}$. Fix an integer $M>0$ such that $\frac{2-\m'}{2}(M+1)>N$.
We note that $p-\l z\in S_{D,\mathrm{unif}}((p+\l),g)$ and $a_{\l,z,0}\in S_{D,\mathrm{unif}}((p+\l)^{-1},g)$ by Lemma \ref{parapsymest}. By \eqref{compasyexp}, the symbol
\begin{align*}
(p-\l z)\# a_{\l,z,0}(x,\x)-\chi_{\l}(x)-\sum_{j=1}^{M}\frac{i^j}{j!}A(D)^j((p(x,\x)-\l z) a_{\l,z,0}(y,\y))|_{x=y,\x=\y}
\end{align*}
belongs to $S_{D,\mathrm{unif}}(\s_g^{M+1},g)$. Since $A(D)$ is a first-order differential operator with constant coefficients, there exist $b_{\l,0,j}\in C^{\infty}(\re^{2n})$ (for $j=0,\hdots,M$) supported in $\supp \chi_{\l}$ and independent of $z$ such that
\begin{align*}
\sum_{j=1}^{M}\frac{i^j}{j!}A(D)^j((p(x,\x)-\l z) a_{\l,z,0}(y,\y))|_{x=y,\x=\y}=\sum_{j=1}^{M+1}\frac{b_{\l,0,j}(x,\x)}{(p(x,\x)-\l z)^j},
\end{align*}
which belongs to $S_{D,\mathrm{unif}}(\s_g,g)$.

Next, we set 
\begin{align*}
a_{\l,z,1}:=-(p-\l z)^{-1}\sum_{j=1}^{M+1}\frac{b_{\l,0,j}}{(p-\l z)^j}=-\sum_{j=0}^{M+1}\frac{b_{\l,0,j}}{(p-\l z)^{j+1}}\in S_{D,\mathrm{unif}}((p+\l)^{-1}\s_g,g).
\end{align*}
Since $b_{\l,0,j}$ is supported in $\chi_{\l}$, we have $a_{\l,z,1}\in S_{D,\mathrm{unif}}((p+\l)^{-1}\s_g,g)$ by Lemma \ref{parapsymest}. 
By using \eqref{compasyexp} again, we can find $b_{\l,1,j}\in C^{\infty}(\re^{2n})$ (for $j=0,\hdots,2M+1$) supported in $\supp \chi_{\l}$ and independent of $z$ such that $\sum_{j=0}^{2M+1}\frac{b_{\l,1,j}}{(p-\l z)^j}\in S_{D,\mathrm{unif}}(\s_g^2,g)$ and
\begin{align*}
(p-\l z)\# (a_{\l,z,0}+a_{\l,z,1}) -\chi_{\l}-\sum_{j=0}^{2M+1}\frac{b_{\l,1,j}}{(p-\l z)^j}\in& S_{D,\mathrm{unif}}(\s_g^{M+1},g)
\end{align*}
Repeating this procedure, we can find $L_M>0$ and 
\begin{align*}
&\{b_{\l,\ell,j}\}_{j=0}^{(\ell+1) M+\ell}\subset  C^{\infty}(\re^{2n}),\quad \ell=2,\hdots,M
\end{align*}
which is supported in $\supp \chi_{\l}$ and is independent of $z$ such that 
\begin{align*}
a_{\l,z,\ell}:=&(p-\l z)^{-1}\sum_{j=0}^{\ell M+\ell-1}\frac{b_{\l,\ell-1,j}}{(p-\l z)^j}\in S_{D,\mathrm{unif}}((p+\l)^{-1}\s_g^{\ell},g),\quad \ell=2,\hdots,M,\\
a_{\l,z}:=&\sum_{\ell=0}^Ma_{\l,z,\ell}\in S((p+\l)^{-1},g)
\end{align*}
satisfy
\begin{align*}
(p-\l z)\# a_{\l,z}-\chi_{\l} \in S_{D,\mathrm{unif}}(\s_g^{M+1},g)\underbrace{\subset}_{\mathclap{\frac{2-\m'}{2}(M+1)>N}} S_{D,\mathrm{unif}}(\jap{x}^{-N}\jap{\x}^{-N},g).
\end{align*}
We can rewrite $a_{\l,z}=\sum_{j=0}^{J_N}\frac{b_{\l,j}}{(p-\l z)^{j+1}}$ with some $J_N>0$ and $b_{\l,j}\in C^{\infty}(\re^{2n})$ which is supported in $\supp \chi_{\l}$ and is independent of $z$.
This completes the proof.

\end{proof}

Now we have prepared to prove \eqref{Ag2}.

\begin{proof}[Proof of \eqref{Ag2}]

Take $Z>0$ such that $\{\Re z\in \supp \g,\, |\Im z|\leq \jap{\Re z}\}\subset \{|z|\leq Z\}$. For simplicity, we write $1_{\l}=1_{\{\l^{-\e}\leq |x|\leq c\l^{-\frac{1}{\m}}\}}$.

Take $N\gg 1$ and $a_{\l,z}$ be as in Lemma \ref{uncparalarge}. Set $R_{\l,z}:=\chi_{\l}-(P-\l z)\Op(a_{\l,z}) \in  S_{D_Z,\mathrm{unif}}^{-N,-N}(g)$. Due to the identity $\chi_{\l}1_{\l}=1_{\l}$, we have
\begin{align}\label{Resoppara}
(P/\l- z)^{-1}1_{\l}=\l\Op(a_{\l,z})1_{\l}-(P/\l- z)^{-1}R_{\l,z}1_{\l}
\end{align}
for $0<\l\leq 1$ and $z\in\mathbb{C}\setminus \re$. Due to Lemma \ref{uncparalarge} and the formula \eqref{HScauchy}, we can write
\begin{align}\label{Ag2mainvanish}
\pi^{-1}\int_{\mathbb{C}}\pa_{\bar{z}}\tilde{\g}(z)a_{\l,z}dm(z)=\sum_{j=1}^{J_N}\frac{(-1)^{j-1}}{(j-1)!}\l^{-j}b_{\l,j}\g^{(j-1)}(p/\l).
\end{align}
We observe that $b_{\l,j}\g^{(j-1)}(p/\l)=0$ holds if $c>0$ is sufficiently small, where we recall $\supp b_{\l,j}\subset \supp \chi_{\l}\subset \{|x|\leq 2c\l^{-\frac{1}{\m}}\}$ by Lemma \ref{uncparalarge}. If we take $c$ sufficiently small, then 
\begin{align*}
\{|x|\leq 2c\l^{-\frac{1}{\m}}\}\cap \supp \g^{(j-1)}(p/\l)=\emptyset\quad \text{if}\,\, \l>0\,\,\text{is sufficiently small}
\end{align*}
since $|p(x,\x)/\l|\leq Z$ implies $C_1\jap{x}^{-\m}\leq V(x)\leq \l Z$ and hence $C_1Z^{-1}\leq \l \jap{x}^{\m}$ by Assumption \ref{potass} $(ii)$.
Therefore, the term \eqref{Ag2mainvanish} vanishes. 

Thus we obtain
\begin{align*}
\g(P/\l)1_{\l}=-\pi^{-1}\int_{\mathbb{C}}\pa_{\bar{z}}\tilde{\g}(z)(P/\l- z)^{-1}R_{\l,z}1_{\l}dm(z).
\end{align*}
by the Helffer-Sj\"ostrand formula \eqref{HScauchy} and the identity \eqref{Resoppara}. 
By Lemma \ref{uncparalarge}, we have $\jap{x}^L R_{\l,z}\in \Op S_{D_Z,\mathrm{unif}}^{-N,L-N}(g)$ and hence $\|\jap{x}^L R_{\l,z}\jap{x}^{L-N}\|\lesssim 1$ by Proposition \ref{unifsymprop} $(iv)$. This implies
\begin{align}\label{xlowbduse}
\|\jap{x}^L R_{\l,z}1_{\l}\|\lesssim \|\jap{x}^{N-L}1_{\l}\| \lesssim \l^{\e(N-L)},
\end{align}
where this is the only place where we use  $\supp 1_{\l}\subset \{|x|\geq \l^{-\e}\}$ as is mentioned before.
Now Lemma  \ref{Resxpres} implies
\begin{align*}
\|\jap{x}^L(P/\l- z)^{-1}R_{\l,z}1_{\l}\|\leq& \|\jap{x}^L(P/\l- z)^{-1}\jap{x}^{-L}\| \|\jap{x}^L R_{\l,z}1_{\l}\|\\
\lesssim&|\Im z|^{-|L|-1} \l^{\e(N-L)-|L|}.
\end{align*}
Consequently, we obtain $\|\g(H/\l)1_{\l}\|=O(\l^{\e (N-L)-|L|})$ due to the property of the almost analytic extension $\tilde{\g}$. Since $N\gg 1$ is arbitrary, this implies  $\|\g(P/\l)1_{\l}\|=O(\l^{\infty})$, which completes the proof.

\end{proof}

\section{Proof of the resolvent estimates at the low-frequency regimes}\label{sectionMourre}

\subsection{LAP for a rescaled operator}\label{subsecresLAP}

We define a scaling operator $D_{\l}$ by
\begin{align*}
D_{\l}u(x):=\l^{\frac{n}{2\m}}u(\l^{\frac{1}{\m}} x)
\end{align*}
and set
\begin{align}\label{P_hdef}
P_{h}:=\l^{-1}D_{\l}^{-1} PD_{\l},\quad V_{h}(x):=\l^{-1}D_{\l}^{-1}V(x) D_{\l}, \quad h:=\l^{\frac{1}{\m}-\frac{1}{2k}}.
\end{align}

\begin{rem}
For $p_0(D_x)=(-\Delta)^m$, we have $\l^{-1}D_{\l}^{-1} p_0(D_x)D_{\l}=(-h^2\Delta)^m$, which is a semiclassical pseudodifferential operator. On the other hand, if $p_0(D_x)=(-\Delta+1)^m-1$, then $\l^{-1}D_{\l}^{-1} p_0(D_x)D_{\l}$ is not a good semiclassical operator. Moreover, the rescaled potential part $V_h(x)=h^{-\frac{2k\m}{2k-\m}}V(h^{-\frac{2k}{2k-\m}}x)$ is singular with respect to $h$.
\end{rem}

In this subsection, we prove resolvent estimates for the rescaled operator $P_h$.

\begin{thm}\label{scalelap}
Let $N\in\mathbb{Z}_{\geq 1}$ and $\c>N-\frac{1}{2}$. Then there exists $h_0>0$ such that 
\begin{align*}
\|\jap{x}^{-\c}(P_{h}-1\mp i0)^{-N}\jap{x}^{-\c}\|\lesssim h^{-N}
\end{align*}
for $0<h\leq h_0$.
\end{thm}

First, we introduce a conjugate operator $A_h:=\frac{x\cdot hD_x+hD_x\cdot x}{2}$ and prove the bound
\begin{align}\label{A_hLAP}
\|\jap{A_h}^{-\c}(P_{h}-1\mp i0)^{-N}\jap{A_h}^{-\c}\|\lesssim h^{-N},
\end{align}
by using a semiclassical alternative of Mourre theory \cite{N1}. By virtue of \cite[Theorem 1]{N1}, it suffices to check (H1): $D(A_h)\cap D(P_h)$ is dense in $D(P_h)$ with respect to its graph norm; (H2): $\mathrm{ad}_{A_h}^jP_h$ can defined on $D(P_h)\cap D(A_h)$ can be extended to a bounded operator from $D(P_h)$ to $L^2$ inductively in $j\in\mathbb{Z}_{\geq 1}$ ; (H3): $\|\mathrm{ad}_{A_h}^j(P_h+i)^{-1}\|\lesssim h^j$ for each $j\in\mathbb{Z}_{\geq 1}$; (H5:$1$): Setting $J_1=[1/2,2]$, we have $E_{P_h}(J_1)[P_h,iA_h] E_{P_h}(J_1)\gtrsim hE_{P_h}(J_1)$, where $E_{P_h}$ denote the spectral projection of $P_h$.
 We also note that the assumption (H4) in \cite{N1} is not needed under the assumption (H5:$1$) \cite[Remark after Theorem 1]{N1}.

Now we find that (H1) and (H2) are easy to check and hence only need to prove (H3) and (H5:$1$).

\subsection*{Semiclassical $C^{\infty}(A_h)$-estimates}

The next lemma shows that $P_h$ and $A_h$ satisfy (H3). For later use, we prove a bit stronger result.

\begin{lem}\label{scaleB_jbdd}
Set $B_j:=\mathrm{ad}_{A_h}^jP_h$.

\noindent$(i)$ For each integer $j\geq 1$, $\|B_j(P_h+1)^{-1}\|\lesssim h^j$ for $0<h\leq 1$.

\noindent$(ii)$ For each integer $j,\ell\geq 1$, 
\begin{align*}
\|(P_h+1)^{\ell}\mathrm{ad}_{A_h}^j(P_h-z)^{-\ell}\|\lesssim h^j|\Im z|^{-j-k}
\end{align*}
for $0<h\leq 1$ and $z\in\mathbb{C}\setminus \re$. In particular, $P_h$ and $A_h$ satisfy (H3).

\end{lem}

\begin{rem}
\cite[Appendix]{M} states that the part $(i)$ can be proved by the estimates $B_j\leq C_jh^jP_h$. However, this just implies $\|(P_h+1)^{-\frac{1}{2}}B_j(P_h+1)^{-\frac{1}{2}}\|\leq C_jh^j$ and does not seem sufficient to prove the bound for $B_j(P_h+1)^{-1}$. To prove the part $(i)$, we need a more precise mapping property of $P_h^{-1}$ (or $P^{-1}$) studied in Theorem \ref{Rieszthm}. For $P=-\Delta+V(x)$, we can use \cite[Thorem 5.3]{Y2} (see Remark \ref{Riszthmrem} $(2)$) instead and the advanced pseudodifferential technique is not needed there.

\end{rem}

\begin{proof}
\noindent$(i)$ By scaling, we have
\begin{align*}
\|B_j(P_h+1)^{-1}\|=&\|D_{\l} B_j(P_h+1)^{-1}D_{\l}^{-1}\|=\|\mathrm{ad}_{hA}^j(\l^{-1}P)(\l^{-1}P+1)^{-1}\|\\
=&h^j\|\mathrm{ad}_{A}^j(P)(P+\l)^{-1}\|.
\end{align*}
On the other hand, the spectral theorem implies $\|P(P+\l)^{-1}\|\leq \sup_{E\in[0,\infty)}\frac{E}{E+\l}\leq 1$.
Thus it suffices to prove $\mathrm{ad}_{A}^j(P)P^{-1}\in B(L^2)$.

We note that $(\mathrm{ad}_{A}^jV)$ is just a multiplication operator and we denote the corresponding function by $W$. By Assumption \ref{potass} $(i)$, we have $|W(x)|\lesssim \jap{x}^{-\m}$. By Theorem \ref{Rieszthm} $(iii)$ with $s=t=0$ and $N=1$, we have $(W=)(\mathrm{ad}_{A}^jV)P^{-1}\in B(L^2)$. To prove $(\mathrm{ad}_{A}^jp_0(D_x))P^{-1}\in B(L^2(\re^n))$, we observe $(\mathrm{ad}_{A}^jp_0(D_x))=((\x\cdot\pa_{\x})^jp_0)(D_x)$. By Assumption \ref{symbolass} and Taylor's theorem, it turns out that $(\x\cdot\pa_{\x})^jp_0(\x)p_0(\x)^{-1}$ is a well-defined $L^{\infty}$-function on $\re^n$. Hence we obtain $\|((\x\cdot\pa_{\x})^jp_0)(D_x)P^{-1} \|\lesssim \|p_0(D_x)P^{-1}\|$. We write $p_0(D_x)P^{-1}=I+VP^{-1}$ and note $VP^{-1}\in B$ by Theorem \ref{Rieszthm} $(iii)$ with $s=t=0$ and $N=1$. We have proved $(\mathrm{ad}_{A}^jp_0(D_x))P^{-1}\in B(L^2(\re^n))$ in this way. This completes the proof of $(i)$.

\noindent$(ii)$ 
First, we observe
\begin{align}\label{adformulainv}
\mathrm{ad}^j_{C}D^{-\ell}=\sum_{\substack{m_1+m_2+\hdots+m_{s}=j,\\ 1\leq s\leq j,\,\, m_1,\hdots m_{s}\geq 1}}c(m_1,m_2,\hdots,m_{s})D^{-\ell}\prod_{a=1}^{s}\left((\mathrm{ad}_C^{m_a}D)  D^{-1}\right)
\end{align}
for any integers $j,k\geq 1$ and linear operators $C,D$, where constants $c_{jk}(m_1,m_2,\hdots,m_{\ell})\in \re$ are independent of $C,D$. This can be proved by induction on $j,k$ and we omit its proof here.

Applying \eqref{adformulainv} with $C=A_h$ and $D=P_h-z$, we have
\begin{align*}
&\|(P_h+1)^{\ell}\mathrm{ad}_{A_h}^j(P_h-z)^{-\ell}\|\\
&\lesssim \sum_{\substack{m_1+m_2+\hdots+m_{s}=j,\\ 1\leq s\leq j,\,\, m_1,\hdots m_{s}\geq 1}}\|(P_h+1)^{\ell}(P_h-z)^{-\ell}\|\cdot \prod_{a=1}^{s} \|(\mathrm{ad}_{A_h}^{m_a}P_h)(P_h-z)^{-1}\|.
\end{align*}
By the spectral theorem, we have $\|(P_h+1)^{\ell}(P_h-z)^{-\ell}\|\lesssim |\Im z|^{-\ell}$. Therefore, the part $(i)$ yields
\begin{align*}
\|(P_h+1)^{\ell}\mathrm{ad}_{A_h}^j(P_h-z)^{-\ell}\|\lesssim& |\Im z|^{-\ell}\sum_{\substack{m_1+m_2+\hdots+m_{s}=j,\\ 1\leq s\leq j,\,\, m_1,\hdots m_{s}\geq 1}} \prod_{a=1}^{s} \left(\|(\mathrm{ad}_{A_h}^{m_a}P_h)(P_h-i)^{-1}\||\Im z|^{-1} \right)\\
\lesssim&|\Im z|^{-\ell}\sum_{\substack{m_1+m_2+\hdots+m_{s}=j,\\ 1\leq s\leq j,\,\, m_1,\hdots m_{s}\geq 1}} \prod_{a=1}^{s} \left(h^{m_a}|\Im z|^{-1} \right)\lesssim h^j|\Im z|^{-j-\ell}.
\end{align*}
This completes the proof.

\end{proof}

\subsection*{Mourre inequality}

Next, we check the semiclassical Mourre inequality (H5:$1$).

\begin{lem}
Let $J_1=[\frac{1}{2},2]$. Then there exists $a_0>0$ and $h_0>0$ such that
\begin{align}\label{Mouinprop}
E_{P_{h}}(J_1)[P_{h},iA_{h}]E_{P_{h}}(J_1)\geq a_0hE_{P_h}(J_1)^2\quad \text{for}\quad 0<h<h_0.
\end{align}

\end{lem}

\begin{proof}
Define $A:=h^{-1}A_h=\frac{x\cdot hD_xL+hD_x\cdot x}{2}$. Due to Theorem \ref{Agthm} and scaling, there exists $c>0$ such that $\|E_{P_h}(J_1)1_{|x|\leq c}\|=O(h^{\infty})$.

First, we prove
\begin{align}\label{MouFrac}
[P,iA]\geq  a_0P-C1_{|x|\leq c\l^{-\frac{1}{\m}}},\quad 0<\l\ll 1
\end{align}
with a constant $a_0>0$. By Assumption \ref{potass}, there exists a constant $a_1>0$ such that $a_1V(x)+x\cdot\pa_xV(x)\leq 0$ for $|x|>R_0$. Then we have
\begin{align}\label{Moupot}
-a_1V(x)-x\cdot\pa_xV(x)\geq -C1_{|x|\leq c\l^{-\frac{1}{\m}}}
\end{align}
with a constant $C>0$ if $\l>0$ (and $h=\l^{\frac{1}{\m}-\frac{1}{2k}}$) is sufficiently small. On the other hand, Assumption \ref{symbolass} $(ii)$ implies $(\x\cdot \pa_{\x}p_0)(\x)\geq a_2p_0(\x)$ with a constant $a_2>0$. Setting $a_0:=\min(a_1,a_2)$, we have $[P,iA]=(\x\cdot \pa_{\x}p_0)(D_x)-x\cdot \pa_xV(x)\geq a_0p_0(D_x)-x\cdot \pa_xV(x)=a_0P-a_0V(x)-x\cdot \pa_xV(x)\geq a_0P-a_1V(x)-x\cdot \pa_xV(x)$. Combining this inequality with \eqref{Moupot}, we obtain \eqref{MouFrac}.

 By scaling, we have
\begin{align*}
h^{-1}[P_h,iA_h]\geq a_0P_h-C\l^{-1}1_{|x|\leq c},
\end{align*}
where we use $D_{\l}^{-1}PD_{\l}^{-1}=\l P_h$, $D_{\l}^{-1}AD_{\l}^{-1}=A=h^{-1}A_h$ and $D_{\l}^{-1}1_{|x|\leq c\l^{-\frac{1}{\m}}}D_{\l}^{-1}=1_{|x|\leq c}$. Since $E_{P_h}(I)^2=E_{P_h}(I)$ and  $E_{P_h}(I)1_{|x|\leq \b}=O_{L^2}(h^{\infty})$ and $\|E_{P_h}(J_1)1_{|x|\leq c_0}\|=O(h^{\infty})$, we obtain \eqref{Mouinprop}.

\end{proof}

\subsection*{Change of the weight}

Due to the last two lemmas and the discussion after \eqref{A_hLAP}, we obtain \eqref{A_hLAP}. To prove Theorem \ref{scalelap}, we have to replace the weight $\jap{A_h}^{-\c}$ by $\jap{x}^{-\c}$. We need to be careful to do this step since the potential part $V_h(x)$ is singular in $h$ and the principal part is not always a semiclassical operator.

\begin{lem}\label{Mouweichange}

\noindent$(i)$ For $\ell\in \mathbb{Z}_{\geq 0}$, we define $\ell':=\max(\frac{\ell}{2k},\frac{\ell}{2m})$. Then we have
\begin{align*}
\|\jap{x}^{-\ell}A_h^{\ell}(P_h+1)^{-\ell'}\|\lesssim 1 \quad \text{for}\quad 0<h\leq 1.
\end{align*}

\noindent$(ii)$
For $\c>0$, we define $\c'=\max(\frac{\c}{2k},\frac{\c}{2m})$. Then
\begin{align*}
\|\jap{x}^{-\c}(P_h+1)^{-\c'}\jap{A_h}^{\c}\|\lesssim 1\quad \text{for}\quad 0<h\leq 1.
\end{align*}

\end{lem}

\begin{proof}

\noindent$(i)$ By complex interpolation, we may assume $\ell\in 2k\mathbb{Z}_{\geq 0}$.
 Set $A:=h^{-1}A_h$. By scaling and the inequality $\|\jap{\l^{\frac{1}{\m}}x}^{-\ell}\jap{x}^{\ell}\|\lesssim \l^{-\frac{\ell}{\m}}$, 
\begin{align*}
\|\jap{x}^{-\ell}A_h^{\ell}(P_h+1)^{-\ell'}\|=\|\jap{\l^{\frac{1}{\m}}x}^{-\ell}A_h^{s}(\l^{-1}P+1)^{-\ell'}\|
\leq h^{\ell}\l^{-\frac{\ell}{\m}}  \|\jap{x}^{-\ell} A^{\ell}(P/\l+1)^{-\ell'}\|.
\end{align*}
Since $A=\frac{x\cdot D_x+D_x\cdot x}{2}$, we can write 
\begin{align*}
A^{\ell}=\sum_{|\a|\leq \ell}a_{\a}(x)\pa_x^{\a}
\end{align*}
where $a_{\a}(x)$ is a polynomial of degree $|\a|$. Then we have $\|\jap{x}^{-\ell} A^{\ell}(P/\l+1)^{-\ell'}\|\lesssim \sum_{|\a|\leq \ell}\|\jap{x}^{-\ell}a_{\a}(x)\pa_x^{\a}(P/\l+1)^{-\ell'}\|\lesssim \|\jap{x}^{-\ell+|\a|}\pa_x^{\a}(P/\l+1)^{-\ell'} \|$. Thus it suffices to prove $\|\jap{x}^{-\ell+|\a|}\pa_x^{\a}(P/\l+1)^{-\ell'} \|\lesssim \l^{\frac{\ell}{2k}}$ for $|\a|\leq \ell$ (here we recall $h=\l^{\frac{1}{\m}-\frac{1}{2k}}$). 
 
Take $\g\in C_c^{\infty}(\re)$ such that $\g(s)=1$ for $|s|\leq 1$ and $\g(s)=0$ for $|s|\geq 2$. Setting $\g_1:=1-\g$, we have 
\begin{align*}
&\|\jap{x}^{-\ell+|\a|}\pa_x^{\a}(P/\l+1)^{-\ell'} \|\\
&\leq \|\jap{x}^{-\ell+|\a|}\pa_x^{\a}(P/\l+1)^{-\ell'}\g(P) \|+\|\jap{x}^{-\ell+|\a|}\pa_x^{\a}(P/\l+1)^{-\ell'}\g_1(P) \|=:I_1+I_2.
\end{align*}
Since $\ell'\geq \frac{\ell}{2k}$, we have
\begin{align*}
I_1\leq\|\jap{x}^{-\ell+|\a|}\pa_x^{\a}(P/\l+1)^{-\frac{\ell}{2k}}\g(P) \|=\l^{\frac{\ell}{2k}}\|\jap{x}^{-\ell+|\a|}\pa_x^{\a}(P+\l)^{-\frac{\ell}{2k}}\g(P) \|.
\end{align*}
Now we observe that $\jap{x}^{-\ell+|\a|}\in \Op S_{J,\mathrm{unif}}(\jap{x}^{-\ell+|\a|},g)$, $\pa^{\a}\in \Op S_{J,\mathrm{unif}}(\jap{\x}^{\frac{\max(k-m,0)}{k}|\a|}(p+\l)^{\frac{|\a|}{2k}},g)$, $(P+\l)^{-\frac{\ell}{2k}}\in\Op S_{J,\mathrm{unif}}((p+\l)^{-\frac{\ell}{2k}},g)$, and $\g(P)\in \Op S_{J,\mathrm{unif}}(\jap{\x}^{-L},g)$ for all $L>0$ by Lemma \ref{lemweight} $(ii)$, Theorem \ref{Rieszthm} $(ii)$, and Lemma \ref{Pfuncpse}. Then
\begin{align*}
\jap{x}^{-\ell+|\a|}\pa_x^{\a}(P+\l)^{-\frac{\ell}{2k}}\g(P)\in \Op S_{J,\mathrm{unif}}(\jap{x}^{-\ell+|\a|}(p+\l)^{\frac{|\a|-\ell}{2k}},g)\subset \Op S_{J,\mathrm{unif}}(1,g)
\end{align*}
by the composition formula (Proposition \ref{unifsymprop} $(i)$).
Hence we obtain $I_1\leq \l^{\frac{\ell}{2k}}\|\jap{x}^{-\ell+|\a|}\pa_x^{\a}(P+\l)^{-\frac{\ell}{2k}}\g(P)\|\lesssim \l^{\frac{\ell}{2k}}$ by Proposition \ref{unifsymprop} $(iv)$. On the other hand, by virtue of the support property of $\g_1$, we have $\|(P+1)^{\ell'}\g_1(P)(P+\l)^{-\ell'}\|\lesssim 1$. Therefore
\begin{align*}
I_2=\l^{\ell'}\|\jap{x}^{-\ell+|\a|}\pa_x^{\a}(P+\l)^{-\ell'}\g_1(P)  \|\lesssim \l^{\ell'}\|\pa_x^{\a}(P+1)^{-\ell'} \|\lesssim \l^{\ell'},
\end{align*}
where we use $|\a|\leq \ell$ and $(P+1)^{-\ell'}\in \Op S^{-2m\ell',0}\subset \Op S^{-\ell,0}$. Since $\ell'\geq \frac{\ell}{2k}$ and $0<\l\leq 1$, we obtain $I_2\lesssim \l^{\frac{\ell}{2k}}$. This completes the proof.

\noindent$(ii)$
By complex interpolation, we may assume $\c=2j$ with an integer $j\geq 1$. We set $(2j)':=\max(\frac{2j}{2k},\frac{2j}{2m})$.

We only need to show $\|\jap{x}^{-2j}(P_h+1)^{-(2j)'}A_h^{2j}\|\lesssim 1$.
We observe
\begin{align}\label{adformulainv2}
DC^j=\sum_{\ell=0}^jc_j(\ell)C^k\mathrm{ad}^{j-\ell}_CD
\end{align}
for each $j=1,2,\hdots$ and linear operators $C,D$, where constants $c_j(k)\in \re$ are independent of $C,D$. We prove \eqref{adformulainv2} by induction. The case $j=1$ is easy to prove. If it holds for $j$, then
\begin{align*}
DC^{j+1}=&\sum_{\ell=0}^jc_j(\ell)C^{\ell}(\mathrm{ad}^{j-\ell}_CD)C=\sum_{\ell=0}^jc_j(\ell)\left(C^{\ell+1}(\mathrm{ad}^{j-\ell}_CD)-C^k\mathrm{ad}_C^{j-\ell+1}D \right)\\
=&\sum_{\ell=1}^{j+1}c_j(\ell-1)C^{\ell}(\mathrm{ad}^{j+1-\ell}_CD)-\sum_{\ell=0}^jc_j(\ell)C^{\ell}\mathrm{ad}_C^{j+1-\ell}D.
\end{align*}
Hence, we set $c_{j+1}(j+1):=c_j(j), c_{j+1}(0):=c_j(0)$ and $c_{j+1}(\ell):=c_j(\ell-1)-c_j(\ell)$ for $\ell=1,\hdots,j$, then \eqref{adformulainv2} holds for $j+1$.

By virtue of \eqref{adformulainv2}, we obtain $(P_h+1)^{-(2j)'}A_h^{2j}=\sum_{\ell=0}^{2j}c_{2j}(\ell)A_h^{\ell}\mathrm{ad}^{2j-\ell}_{A_h}(P_h+1)^{-(2j)'}$.
Therefore, the first part $(i)$ and Lemmas \ref{scaleB_jbdd} imply
\begin{align*}
\|\jap{x}^{-2j}(P_h+1)^{-(2j)'}A_h^{2j}\|\lesssim& \sum_{\ell=0}^{2j}\|\jap{x}^{-2j}A_h^\ell(P_h+1)^{-\ell'}\| \|(P_h+1)^{\ell'}\mathrm{ad}^{2j-\ell}_{A_h}(P_h+1)^{-(2j)'}\|\\
\lesssim& \sum_{\ell=0}^{2j}\|\jap{x}^{-\ell}A_h^{\ell}(P_h+1)^{-\ell'}\| \|(P_h+1)^{j'}\mathrm{ad}^{2j-\ell}_{A_h}(P_h+1)^{-(2j)'}\|\\
\lesssim&1.
\end{align*}

\end{proof}

Now we have prepared to prove Theorem \ref{scalelap}.

\begin{proof}[Proof of Theorem \ref{scalelap}]
Let $\g\in C_c^{\infty}((\frac{1}{2},2))$ such that $\g(s)=1$ near $s=1$. By the spectral theorem, it suffices to prove $\|\jap{x}^{-\c}\g(P_h)(P_{h}-1\mp i0)^{-N}\jap{x}^{-\c}\|\lesssim Ch^{-N}$.
We define $\tilde{\g}(s):=(s+1)^{2\c'}\g(s) \in C_c^{\infty}((\frac{1}{2},2))$, where $\c'$ is defined in Lemma \ref{Mouweichange}.
Due to \eqref{A_hLAP} and the spectral theorem, we have
\begin{align}\label{A_hLAP2}
\|\jap{A_h}^{-\c}\tilde{\g}(P_h)(P_{h}-1\mp i0)^{-N}\jap{A_h}^{-\c}\|\lesssim Ch^{-N}
\end{align}
and hence
\begin{align*}
&\|\jap{x}^{-\c}\g(P_h)(P_{h}-1\mp i0)^{-N}\jap{x}^{-\c}\|\\
&\leq \|\jap{x}^{-\c}(P_h+1)^{-\c'}\jap{A_h}^{\c}\|\cdot \|\jap{A_h}^{-\c}\tilde{\g}(P_h)(P_{h}-1\mp i0)^{-N}\jap{A_h}^{-\c}\|\cdot \|\jap{A_h}^{\c}(P_h+1)^{-\c'}\jap{x}^{-\c}\|\\
&\lesssim h^{-N}
\end{align*}
where we use \eqref{A_hLAP2} and Lemma \ref{Mouweichange} $(ii)$. This completes the proof.

\end{proof}

\subsection{Proof of LAP for the original operator at the low-frequency regimes}

Now we prove the resolvent estimates at the low-frequency regimes.

\begin{thm}\label{genmainthm}
Let $N\in\mathbb{Z}_{>0}$ and $\c>\max(N-1/2,N\left(\frac{1}{2}+\frac{2k-1}{4k}\m\right) )$. Then there exists $Z_0>0$ such that
\begin{align*}
\sup_{z\in\overline{\mathbb{C}}_{\pm},\,\,0<|z|\leq Z_0}\|\jap{x}^{-\c}R_{\pm}(z)^N\jap{x}^{-\c}\|<\infty.
\end{align*}
Moreover, the map $\{z\in\mathbb{C}\mid 0<|z|\leq Z_0\}\ni z\mapsto \jap{x}^{-\c}R_{\pm}(z)^N\jap{x}^{-\c}\in B(L^2)$ is H\"older continuous. In particular, $\jap{x}^{-\c}R_{\pm}(0)^N\jap{x}^{-\c}$ exists in $B(L^2)$. In addition, we have $R_+^N(0)=R_-^N(0)$.

\end{thm}

The remaining part is just to adapt the argument in \cite{N2}.

\begin{lem}\label{scaleretest}
Let $\g\in C_c^{\infty}(1/2,3/2)$, $N\in\mathbb{Z}_{>0}$ and $\c>N-1/2$. Then there exists $Z_0>0$ such that
\begin{align*}
\|\jap{x}^{-\c}R_{\pm}(\l)^N\g(P/\l)\jap{x}^{-\c}\|\lesssim \l^{\a(N,\c)}
\end{align*}
for $\l\in (0,Z_0]$, where we set $\a(N,\c)=\frac{2\c}{\m}-N\left(\frac{1}{\m}+\frac{2k-1}{2k}\right)$.

\end{lem}

\begin{proof}
We note $\jap{x}\sim (1+|x|)$. By scaling, we have 
\begin{align*}
\|\jap{x}^{-\c}R_{\pm}(\l)^{-N}\g(P/\l)\jap{x}^{-\c}\|=&\|\jap{\l^{-\frac{1}{\m}}x}^{-\c}(\l P_h-\l\mp i0)^{-N}\g(P_h)\jap{\l^{-\frac{1}{\m}}x}^{-\c}\|\\
\lesssim&\l^{\frac{2\c}{\m}-N}\|(\l^{\frac{1}{\m}} +|x|)^{-\c}(P_h-1\mp i0)^{-N}\g(P_h)(\l^{\frac{1}{\m}} +|x|)^{-\c}\|.
\end{align*}
Thus it suffices to prove
\begin{align*}
\|(\l^{\frac{1}{\m}} +|x|)^{-\c}(P_h-1\mp i0)^{-N}\g(P_h)(\l^{\frac{1}{\m}} +|x|)^{-\c}\|\lesssim h^{-N},
\end{align*}
where $P_h$ is defined in the last subsection and $h=\l^{\frac{1}{\m}-\frac{1}{2k}}$.

First, we observe 
\begin{align}\label{Agapp}
\|\jap{x}^{\c}\g(P_h) 1_{\{|x|\leq c\}} \|=O(\l^{\infty})
\end{align}
for sufficiently small $c>0$, which follows from scaling, Theorem \ref{Agthm}, and an inequality $\|\jap{\l^{\frac{1}{\m}}x}^{\c}\jap{x}^{-\c}\|\leq 1$.

We write $A=(P_h-1\mp i0)^{-N}$, $B=(\l^{\frac{1}{\m}} +|x|)^{-\c}$, and
\begin{align*}
BA\g(P_h)B=&B1_{\{|x|\leq c\}}A\g(P_h)1_{\{|x|\leq c\}}B+B1_{\{|x|\leq c\}}\g(P_h)A1_{\{|x|> c\}}B\\
&+B1_{\{|x|> c\}}A\g(P_h)1_{\{|x|\leq c\}}B+B1_{\{|x|> c\}}A\g(P_h)1_{\{|x|> c\}}B.
\end{align*}
By Theorem \ref{scalelap}, we have $\|\jap{x}^{-\c}A\jap{x}^{-\c}\|\lesssim h^{-N}$. Moreover,
$\|\jap{x}^{\c}\g(P_h)1_{\{|x|\leq c\}}B\|=O(\l^{\infty})$ and $\|B1_{\{|x|\leq c\}}\g(P_h)\jap{x}^{\c}\|=O(\l^{\infty})$ due to \eqref{Agapp}. Finally, $\|\jap{x}^{\c}1_{\{|x|> c\}}B\|\lesssim 1$ and $\|B1_{\{|x|> c\}}\jap{x}^{\c}\|\lesssim 1$. Combining these estimates, we obtain $\|BA\g(P_h)B\|\lesssim h^{-N}$, which completes the proof.

\end{proof}

Now we consider the derivative of the spectral projection:
\begin{align}\label{spresid}
E'(\l)=\frac{1}{2\pi i}\left(R_+(\l)-R_-(\l)\right),
\end{align}
where this identity follows from the absolutely continuity of the spectrum $\s(P)$ and Stone's theorem.
We note $\frac{d^{N-1}}{d\l^{N-1}}E'(\l)=\frac{(N-1)!}{2\pi i}\left(R_+(\l)^{N}-R_-(\l)^{N}\right)$.

\begin{cor}\label{spmesest}
Let $N\in\mathbb{Z}_{\geq 1}$ and $\c>N-1/2$. Then there exists $Z_0>0$ such that the following holds:

\noindent$(i)$ If $\c\geq N(\frac{1}{2}+\frac{2k-1}{4k}\m)$, then $\jap{x}^{-\c}\frac{d^{N-1}}{d\l^{N-1}}E'(\l)\jap{x}^{-\c}$ is bounded in $B(L^2)$ with respect to $0< \l\leq Z_0$.

\noindent$(ii)$ If $\c> N(\frac{1}{2}+\frac{2k-1}{4k}\m)$, $\jap{x}^{-\c}\frac{d^{N-1}}{d\l^{N-1}}E'(\l)\jap{x}^{-\c}$ is H\"older continuous in $B(L^2)$ with respect to $0< \l\leq Z_0$. Moreover,
\begin{align*}
\left\|\jap{x}^{-\c}\frac{d^{N-1}}{d\l^{N-1}}E'(\l)\jap{x}^{-\c}\right\|\lesssim \l^{\a(N,\c)}\quad \text{for}\quad 0<\l\leq Z_0,
\end{align*}
where $\a(N,\c)=\frac{2\c}{\m}-N\left(\frac{1}{\m}+\frac{2k-1}{2k}\right)$.

\end{cor}

\begin{proof}
The part $(i)$ follows from Lemma \ref{scaleretest} and the fact that $\g(P/\l)E'(\l)=E'(\l)$ if $\g=1$ near $1$. The part $(ii)$ follows from the part $(i)$ and complex interpolation. 
\end{proof}

\begin{proof}[Proof of Theorem \ref{genmainthm}]
Now Theorem \ref{genmainthm} follows from Corollary \ref{spmesest}, the formula $(H-z)^{-1}1_{[0,1]}(H)=\int_{0}^1E'(\l)(\l-z)^{-1}d\l$ (which is a consequence of the spectrum theorem), and the inequality
\begin{align*}
\sup_{0<\e\leq 1}\left|\int_{-1}^1\frac{\f(s)}{s\mp i\e}  ds\right|\lesssim& \|\f\|_{C^{\a}},\quad \f\in C_c^{\a}(\re),\quad 0<\a\leq 1.
\end{align*}
The identity $R_+^N(0)=R_-^N(0)$ follows from Proposition \ref{spmesest} $(ii)$: $\frac{d^{N-1}}{d\l^{N-1}}E'(0)=0$.

\end{proof}

\section{High-frequency estimates and time decay of the propagator }

\subsection{High-frequency estimates}\label{subsechigh}

\begin{thm}\label{thmhighlap}
Let $N\in\mathbb{Z}_{\geq 1}$, $\c>N-\frac{1}{2}$ and $Z_0>0$. Then $\jap{x}^{-\c}R_{\pm}(z)^N\jap{x}^{-\c}$ is H\"older continuous in $B(L^2)$ on $\{z\in \overline{\mathbb{C}}_{\pm}\mid |z|\geq Z_0\}$ and
\begin{align}\label{highlong}
\sup_{z\in \overline{\mathbb{C}}_{\pm},\,|z|\geq Z_0}|z|^{\frac{N}{2}\left(2-\frac{1}{m}\right)}\|\jap{x}^{-\c}R_{\pm}(z)^N\jap{x}^{-\c}\| <\infty.
\end{align}

\end{thm}

\begin{rem}
We do not need positive repulsive conditions (Assumption \ref{potass} $(ii), (iii)$) to obtain this result. They are needed for the low-frequency regimes only.
\end{rem}

\begin{proof}
This theorem is an immediate consequence of the standard and semiclassical Mourre theory and the assumption on absence of eigenvalues. Here we just give an outline of the proof. 

By a simple application of the Mourre theory \cite{JMP} and Assumption \ref{abevass}, it suffices to prove $(\ref{highlong})$ for $|z|>>1$. 
 Set $h=|z|^{-\frac{N}{2m}}<<1$ and consider the semiclassical operator $\tilde{P}_h:=h^{2m}P=h^{2m}p_0(D_x)+h^{2m}V$ and $A_h:=\frac{hx\cdot D_x+hD_x\cdot x}{2}$. Then the semiclassical Mourre theory (for example, see \cite[Theorem 1]{N1} ) with Assumptions \ref{symbolass} $(ii)$ and \ref{potass} $(i)$ implies
\begin{align*}
\sup_{|\l|=1,\,\, \l\in \overline{\mathbb{C}}_{\pm}}\|\jap{x}^{-\c}(\tilde{P}_h-\l)^{-N}\jap{x}^{-\c}\|\lesssim \frac{1}{h^N}
\end{align*}
for $h>0$ sufficiently small.  Multiplying $h^{2mN}$ of both sides of this inequality, we obtain $(\ref{highlong})$ for $|z|>>1$. Here, the reason why the semiclassical Mourre estimate \cite[$(\mathrm{H}5;1)$]{N1} holds uniformly in $h$ is that $h^{2m}V$ is sufficiently small when $0<h\ll 1$. We omit the detail.
\end{proof}

\subsection{Local time decay estimates}\label{subsectimedecay}

Now we prove Theorem \ref{fracmainthm} $(ii)$. We follow the argument in \cite[Theorem 4.2]{JMP}, see also \cite[\S 5]{BB}. Since $e^{-itP}$ is unitary, we may assume $|t|\geq 1$. By the general interpolation inequality $\|\jap{x}^{-\theta s}A\jap{x}^{-\theta s}\|\leq \|A\|^{1-\theta}\|\jap{x}^{-s}A\jap{x}^{- s}\|^{\theta}$ for $\theta\in [0,1]$, it suffices to prove that for each $s>0$, there exists $\c>0$ such that
\begin{align}\label{timedecayint}
\|\jap{x}^{-\c}\jap{P}^{-a}e^{-itP}\jap{x}^{-\c}\|=O(\jap{t}^{-s}),
\end{align}
where $a=0$ when $m> 1/2$ and $a>0$ is chosen appropriately when $m\leq 1/2$.

Take an integer $N>s'$ such that $\frac{N}{2}\left(2-\frac{1}{m}\right)>1$ when $m> 1/2$. Define $a=-\frac{N}{2}\left(2-\frac{1}{m}\right)+2$ when $m\leq 1/2$.
We have $\jap{P}^{-a}e^{-itP}=\int_0^{\infty}e^{-it\l}\jap{\l}^{-a} E'(\l) d\l$ due to the spectral theorem. 
By integration by parts, 
\begin{align}
\jap{x}^{-\c}\jap{P}^{-a}e^{-itP}\jap{x}^{-\c}=&-\sum_{j=0}^{N-1}\frac{1}{(it)^{j+1}}\left[e^{-it\l}\jap{\l}^{-a}\jap{x}^{-\c}\frac{d^{j}}{d\l^{j}}E'(\l)\jap{x}^{-\c} \right]_{\l=0}^{\infty}\label{propaintby}\\
&+\frac{1}{(it)^{N}}\int_0^{\infty}e^{-it\l}\jap{\l}^{-a}\jap{x}^{-\c}\frac{d^{N}}{d\l^{N}}E'(\l)\jap{x}^{-\c}d\l.\nonumber
\end{align}

By \eqref{spresid}, Corollary \ref{spmesest}, and Theorem \ref{thmhighlap}, if we take $\c>0$ large enough, we have  $\jap{x}^{-\c}E'(\l)\jap{x}^{-\c}\in C^{N}([0,\infty);B(L^2))$, $\jap{x}^{-\c}\frac{d^{j}}{d\l^j}E'(0)\jap{x}^{-\c}=0$ for $j=0,\hdots,N-1$ and  
\begin{align*}
\left\|\jap{x}^{-\c}\frac{d^{j}}{d\l^{j}}E'(\l)\jap{x}^{-\c}\right\|\lesssim \jap{\l}^{-\frac{j+1}{2}\left(2-\frac{1}{m}\right)}\ \quad \text{for} \quad j=1,\hdots,N.
\end{align*}
There implies that the boundary values appeared in the right hand side of \eqref{propaintby} vanish. Thus we have
\begin{align*}
\|\jap{x}^{-\c}\jap{P}^{-a}e^{-itP}\jap{x}^{-\c}\|\leq& |t|^{-N}\left\|\int_0^{\infty}e^{-it\l}\jap{\l}^{a}\jap{x}^{-\c}\frac{d^{N}}{d\l^{N}}E'(\l)\jap{x}^{-\c}d\l\right\|\\
\lesssim&|t|^{-N}\int_{0}^{\infty}\jap{\l}^{-a}\jap{\l}^{-\frac{N}{2}\left(2-\frac{1}{m}\right)}d\l\lesssim |t|^{-N}.
\end{align*}
Since we have assumed $N>s'$ and $|t|\geq 1$, this proves \eqref{timedecayint}.

\appendix

\section{Absence of eigenvalues under a Virial condition}

In this short section, we give a sufficient condition so that our operator $P$ has no eigenvalues.
The following argument is a more or less standard, see \cite[Theorem 1.11]{FSWY}, \cite[Proposition II.4]{Mo}, or \cite[Theorem XIII.59, Theorem XIII.60]{RS}.

\begin{thm}\label{Virthm}
Let $p_0\in C^{\infty}(\re^n;[0,\infty))$ satisfy Assumption \ref{symbolass}. Suppose that a non-negative function $V\in L^{\infty}(\re^n)\cap C^{1}(\re^n;[0,\infty))$ satisfying $x\cdot \pa_xV\in L^{\infty}(\re^n)$ and $-x\cdot \pa_xV(x)\geq 0$ for all $x\in \re^n$. Then the self-adjoint operator $P=p_0(D_x)+V(x)$ with domain $H^{2m}(\re^n)$ has no eigenvalues.

\end{thm}

\begin{rem}
A typical example of $V$ satisfying the assumptions of this theorem and Assumption \ref{potass} is $V(x)=c\jap{x}^{-\m}$ for $c>0$.
\end{rem}

\begin{proof} 
Since $p_0\geq 0$ and $V\geq 0$, the spectrum $\s(P)$ of $P$ is contained in $[0,\infty)$.

First, we show that $P$ has no positive eigenvalues. Let $\l>0$ be an eigenvalue of $P$ and $u\in L^2(\re^n)\setminus \{0\}$ be an eigenfunction of $p$ associated with $\l$. Define $A=\frac{x\cdot D_x+D_x\cdot x}{2}$. 
By Assumption \ref{symbolass} and our assumptions $V,x\cdot \pa_xV\in L^{\infty}(\re^n)$, the operator $P$ and $A$ satisfy the condition $(a),(b),(c)$ in \cite[1.Definition]{Mo} ($P=H$ in the notation there). Then the Virial theorem (\cite[Proposition II.4]{Mo}) shows that $(u,[P,iA]u)=0$. On the other hand, due to Assumption \ref{symbolass} $(ii)$ and the assumption $-x\cdot \pa_xV(x)\geq 0$, we have $[P,iA]=(\x\cdot \pa_{\x}p_0)(D_x)-x\cdot \pa_xV(x)\geq cp_0(D_x)$ with a constant $c>0$. These implies $(u,p_0(D_x)u)=0$. Since $p_0(\x)\geq 0$ and $\{\x\in \re^n\mid p_0(\x)=0\}=\{0\}$ (due to Assumption \ref{symbolass} $(ii)$), we obtain $u=0$, which is a contradiction.

Finally, we prove that $0$ is not an eigenvalue of $P$. If $u\in L^2(\re^n)$ satisfies $Pu=0$, then $0=(u,Pu)\geq (u,p_0(D_x)u)\geq 0$ by $p_0(\x),V(x)\geq 0$. Thus we obtain $u=0$ in the same way as above.

\end{proof}

\section{Some proof of the results in Subsection \ref{subsectionslvame}}\label{appendixslvamet}

In this appendix, we prove Proposition \ref{propslowvary} and Lemma \ref{lemweight}.
We recall that $g_{(x,\x)}=\jap{x}^{-2}dx^2+\jap{x}^{\m'}\jap{\x}^{-2}d\x^2$, $0<\m'(=\m/k)<2$, and $A=\begin{pmatrix}
0&-I\\
I&0
\end{pmatrix}$.

\subsection*{Proof of Proposition \ref{propslowvary} $(i)$} We need to find $c>0$ such that $g_{(x,\x)}(y,\y)\leq c$ implies $g_{(x,\x)}(z,\z)\lesssim g_{(x+y,\x+\y)}(z,\z)\lesssim g_{(x,\x)}(z,\z)$.
We first observe that
$g_{(x,\x)}(y,\y)\leq c$ is equivalent to $\jap{x}^{-2}|y|^2+\jap{x}^{\m'}\jap{\x}^{-2}|\y|^2\leq c$ and suppose that this inequality holds. If we take $c>0$ small enough, then $ \jap{x+y}^2\sim \jap{x}^2$ and $\jap{\x+\y}^2\sim \jap{\x}^2$ holds. Then
\begin{align*}
\frac{|z|^2}{\jap{x}^2}+\jap{x}^{\m'}\frac{|\z|^2}{\jap{\x}^2}\lesssim  \frac{|z|^2}{\jap{x+y}^2}+\jap{x+y}^{\m'}\frac{|\z|^2}{\jap{\x+\y}^2}\lesssim\frac{|z|^2}{\jap{x}^2}+\jap{x}^{\m'}\frac{|\z|^2}{\jap{\x}^2},
\end{align*}
which is equivalent to $g_{(x,\x)}(z,\z)\lesssim g_{(x+y,\x+\y)}(z,\z)\lesssim g_{(x,\x)}(z,\z)$.

\subsection*{Calculation of $g^{\s}$ and $\s_g$: Proof of Proposition \ref{propslowvary} $(iii)$}

First, we show 
\begin{align}\label{gAcal}
g^{A}_{(x,\x)}=\jap{\x}^2\jap{x}^{-\m'}dx^2+\jap{x}^2d\x^2.
\end{align}
We note $A(z,\z)=(-\z,z)$and $g_{(y,\y)}(A(z,\z))=g_{(y,\y)}(-\z,z)=\jap{y}^{-2}|\z|^2+\jap{y}^{\m'}\jap{\y}^{-2}|z|^2$. By the definition of $g^A$: \eqref{gAdef}, we have
\begin{align*}
g^{A}_{(x,\x)}(y,\y)=&\sup_{(z,\z)\neq (0,0)}\frac{|((y,\y),(z,\z))|^2}{g_{(x,\x)}(A(z,\z))}=\sup_{(z,\z)\neq (0,0)}\frac{|y\cdot z+\y\cdot \z|^2}{\jap{x}^{-2}|\z|^2+\jap{x}^{\m'}\jap{\x}^{-2}|z|^2}\\
=&\sup_{(z,\z)\neq (0,0)}\frac{|y\cdot \jap{\x}\jap{x}^{-\m'/2}z+\y\cdot \jap{x}\z|^2}{|\z|^2+|z|^2}=\sup_{|z|^2+|\z|^2=1}|y\cdot \jap{\x}\jap{x}^{-\m'/2}z+\y\cdot \jap{x}\z|^2
\end{align*}
Its supremum is achieved when $(\jap{\x}\jap{x}^{-\m'/2}y,\jap{x}\y)\parallel (z,\z)$, that is
\begin{align*}
(z,\z)=\frac{1}{\sqrt{\jap{x}^2\jap{x}^{-\m'}|y|^2+\jap{x}^2|\y|^2}}(\jap{\x}\jap{x}^{-\m'/2}y,\jap{x}\y).
\end{align*}
Then
\begin{align*}
g^{A}_{(x,\x)}(y,\y)=\sqrt{\jap{\x}^2\jap{x}^{-\m'}|y|^2+\jap{x}^2|\y|^2}^2=\jap{\x}^2\jap{x}^{-\m'}|y|^2+\jap{x}^2|\y|^2,
\end{align*}
which gives $g^{A}_{(x,\x)}=\jap{\x}^2\jap{x}^{-\m'}dx^2+\jap{x}^2d\x^2$.

Next, we show $\s_g(x,\x)=\jap{x}^{\frac{\m'}{2}-1}\jap{\x}^{-1}$. By the definition of $\s_g$: \eqref{sigmagdef}, we have
\begin{align*}
\s_g(x,\x)^2=&\sup_{(y,\y)\in \re^{2n}\setminus \{0\}}\frac{g_{(x,\x)}(y,\y)}{g^{A}_{(x,\x)}(y,\y)}=\sup_{(y,\y)\in \re^{2n}\setminus\{0\}}\frac{\jap{x}^{-2}|y|^2+\jap{\x}^{-2}\jap{x}^{\m'}|\y|^2 }{\jap{\x}^2\jap{x}^{-\m'}|y|^2+\jap{x}^2|\y|^2}\\
=&\jap{x}^{\m'-2}\jap{\x}^{-2}\sup_{(y,\y)\in \re^{2n}\setminus\{0\}}1=\jap{x}^{\m'-2}\jap{\x}^{-2}.
\end{align*}

\subsection*{Proof of Proposition \ref{propslowvary} $(ii)$} We need to show $g_{(x,\x)}(z,\z)\lesssim (1+g_{(x,\x)}^A(x-y,\x-\y))^Ng_{(y,\y)}(z,\z)$ for some $N>0$. By the definition of $g$ and $g^A$, we have $g_{(x,\x)}^A(x-y,\x-\y)=\jap{\x}^2\jap{x}^{-\m'}|x-y|^2+\jap{x}^2|\x-\y|^2$ and
\begin{align*}
\frac{g_{(x,\x)}(z,\z)}{g_{(y,\y)}(z,\z)}=\frac{\jap{x}^{-2}|z|^2+\jap{\x}^{-2}\jap{x}^{\m'}|\z|^2}{\jap{y}^{-2}|z|^2+\jap{\y}^{-2}\jap{y}^{\m'}|\z|^2}\leq \jap{y}^2\jap{x}^{-2}+\jap{y}^{\m'}\jap{x}^{-\m'}\jap{\y}^2\jap{\x}^{-2}.
\end{align*}
We observe the three elementary inequalities: $\jap{y}^2\jap{x}^{-2}\lesssim (1+\jap{x}^{-\m'}|x-y|^2)^{\frac{2}{2-\m'}}\leq(1+g^A_{(x,\x)}(x-y,\x-\y))^{\frac{2}{2-\m'}}$, $\jap{y}^{\m'}\jap{x}^{-\m'}\lesssim (1+g^A_{(x,\x)}(x-y,\x-\y))^{\frac{\m'}{2-\m'}}$ and $\jap{\y}^2\jap{\x}^{-2}\lesssim (1+|\x-\y|^2)\lesssim (1+g^A_{(x,\x)}(x-y,\x-\y))$, we obtain $g_{(x,\x)}(z,\z)\lesssim (1+g_{(x,\x)}^A(x-y,\x-\y))^Ng_{(y,\y)}(z,\z)$ with $N=\frac{2}{2-\m'}$. This completes the proof of Proposition \ref{propslowvary} $(ii)$.

Now we have finished the proof of Proposition \ref{propslowvary}. We turn to prove Lemma \ref{lemweight}.

\subsection*{Uniform $g$-continuity of $p+\l$}

First, we prove the uniform $g$-continuity of $p+\l$, that is to find $c>0$ such that $g_{(x,\x)}(y,\y)\leq c$ implies $p(x,\x)+\l\lesssim p(x+y,\x+\y)+\l\lesssim p(y,\y)+\l$ uniformly in $\l\in [0,1]$. Clearly, we may assume $\l=0$.

 Suppose that $g_{(x,\x)}(y,\y)\leq c\Leftrightarrow \jap{x}^{-2}|y|^2+\jap{x}^{\m'}\jap{\x}^{-2}|\y|^2\leq c$ for sufficiently small $c>0$. Then we have $\jap{x+y}\sim \jap{x}$ and
\begin{align*}
|\x|^2\leq 2|\x+\y|^2+2|\y|^2\leq& 2|\x+\y|^2+2c\jap{x}^{-\m'}+2c|\x|^2.
\end{align*}
If $c>0$ is sufficiently small, we have $|\x|^2\lesssim |\x+\y|^2+\jap{x}^{-\m'}$ and hence
\begin{align*}
|\x|^2+\jap{x}^{-\m'}\lesssim |\x+\y|^2+\jap{x+y}^{-\m'},
\end{align*}
which implies $p(x,\x)\lesssim p(x+y,\x+\y)$ by Assumptions \ref{symbolass} and \ref{potass}. The inequality $p(x+y,\x+\y)\lesssim p(y,\y)$ is similarly proved.

\subsection*{Uniform $(A,g)$-temperateness of $p+\l$}
By \eqref{gAcal}, we have $g^{A}_{(x,\x)}((x-y,\x-\y))=1+\jap{\x}^2\jap{x}^{-\m'}|x-y|^2+\jap{x}^2|\x-\y|^2$. Thus it suffices to prove
\begin{align*}
p_0(\x)+\jap{x}^{-\m}+\l\lesssim (1+\jap{\x}^2\jap{x}^{-\m'}|x-y|^2+\jap{x}^2|\x-\y|^2 )^N(p_0(\y)+\jap{y}^{-\m}+\l)
\end{align*}
since $V(x)\sim \jap{x}^{-\m}$ by Assumption \ref{potass}. Clearly, we may assume $\l=0$.

To prove this, we use the following lemmas.

\begin{lem}\label{xylemma}
For $x,y\in \re^n$, $\jap{x}^{2k}\jap{y}^{-\m}<2^{-k}$ implies $\jap{y}\lesssim |x-y|$.
\end{lem}

\begin{proof}

The assumption $0<\m<2k$ implies $1+|x|^2\leq \frac{1}{2}(1+|y|^2)$ and hence $1\leq |y|$ and $2|x|^2\leq |y|^2$. Thus we have $\jap{y}\lesssim |y|\lesssim |y|-|x|\leq |x-y|$.

\end{proof}

\begin{lem}\label{xtemp}
For $k\in \re$,
$\jap{x}^{k}$ is $(A,g)$-temperate. 
\end{lem}

\begin{proof}
Since $g^{A}_{(x,\x)}((x-y,\x-\y))=1+\jap{\x}^2\jap{x}^{-\m'}|x-y|^2+\jap{x}^2|\x-\y|^2$ by \eqref{gAcal}, it suffices to prove
\begin{align}\label{xtempgmu1}
\jap{x}^{k}\lesssim  (1+\jap{\x}^2\jap{x}^{-\m'}|x-y|^2)^{N}\jap{y}^{k}
\end{align}
with $N=\frac{1}{2-\m'}|k|$.
We note that this inequality is trivial for $k=0$.

\noindent$(i)$ First, we assume $k>0$. If $\jap{x}\lesssim \jap{y}$ holds, then \eqref{xtempgmu1} follows from  the trivial inequality $1\leq 1+\jap{\x}^2\jap{x}^{-\m'}|x-y|^2$.

Thus we may assume $\jap{y}\lesssim \jap{x}$. If either $|x-y|\leq \frac{1}{2}|x|$ or $|x|\leq 1$ holds, then we have $\jap{x}\lesssim \jap{y}$, which implies \eqref{xtempgmu1} as noted just before. If both $|x-y|> \frac{1}{2}|x|$ and $|x|\geq 1$ hold, then $\jap{x}\lesssim |x-y|$ and hence
\begin{align*}
\jap{x}^k\lesssim& \jap{x}^{-\frac{\m'}{2-\m'}k}|x-y|^{k+\frac{\m'}{2-\m'}k}=\left(\jap{x}^{-\m'}|x-y|^2 \right)^{\frac{1}{2-\m'}k}\\
\leq&(1+\jap{\x}^2\jap{x}^{-\m'}|x-y|^2)^{\frac{1}{2-\m'}k}\jap{y}^{k}
\end{align*}
since $1\leq \jap{y}^k$.

\noindent$(ii)$ Next, we assume $k<0$. If $\jap{y}\lesssim \jap{x}$ holds, then \eqref{xtempgmu1} follows from  the trivial inequality $1\leq 1+\jap{\x}^2\jap{x}^{-\m'}|x-y|^2$.

Thus we may assume $\jap{x}\lesssim \jap{y}$. If either $|x-y|\leq \frac{1}{2}|y|$ or $|y|\leq 1$ holds, then we have $\jap{y}\lesssim \jap{x}$, which implies \eqref{xtempgmu1} as noted just before. If both $|x-y|> \frac{1}{2}|y|$ and $|y|\geq 1$ hold, then $\jap{y}\lesssim |x-y|$ and hence
\begin{align*}
\jap{y}^{-k}=\jap{y}^{|k|}\underbrace{\lesssim}_{\jap{y}\lesssim |x-y|}& \jap{y}^{-\frac{\m'}{2-\m'}|k|}|x-y|^{|k|+\frac{\m'}{2-\m'}|k|}\underbrace{\lesssim}_{\jap{x}\lesssim \jap{y}} \jap{x}^{-\frac{\m'}{2-\m'}|k|}|x-y|^{|k|+\frac{\m'}{2-\m'}|k|}\\
\underbrace{\lesssim}_{\jap{y}\lesssim |x-y|}&\left(\jap{x}^{-\m'}|x-y|^2 \right)^{\frac{1}{2-\m'}|k|}\underbrace{\lesssim}_{k<0,\,\, 1\leq \jap{x}^{-k}} \left(1+\jap{x}^{-\m'}|x-y|^2 \right)^{\frac{1}{2-\m'}|k|}\jap{x}^{-k},
\end{align*}
which proves \eqref{xtempgmu1}.

\end{proof}

Due to Lemma \ref{xtemp} , it suffices to prove
\begin{align*}
p_0(\x)\lesssim (1+\jap{\x}^2\jap{x}^{-\m'}|x-y|^2+\jap{x}^2|\x-\y|^2 )^N(p_0(\y)+\jap{y}^{-\m})
\end{align*}
with some $N>0$.

\begin{lem}
For $|\x|\leq 1$, we have
\begin{align*}
p_0(\x)\lesssim (1+\jap{\x}^2\jap{x}^{-\m'}|x-y|^2+\jap{x}^2|\x-\y|^2 )^k(p_0(\y)+\jap{y}^{-\m}).
\end{align*}
\end{lem}

\begin{proof}

\noindent$(i)$
First, we consider the case where $|\y|\geq 1/2$ holds. Due to Assumption \ref{symbolass}, we have $p_0(\y)\gtrsim |\y|^{2m}$. Hence the trivial inequality $1\leq (1+\jap{\x}^2\jap{x}^{-\m}|x-y|^2+\jap{x}^2|\x-\y|^2 )^k$ shows 
\begin{align*}
&p_0(\x)\underbrace{\lesssim}_{\text{Assumption \ref{potass}}} |\x|^{2k}\leq 1\leq |\y|^{2m}\lesssim p_0(\y)\\&\leq (1+\jap{\x}^2\jap{x}^{-\m'}|x-y|^2+\jap{x}^2|\x-\y|^2 )^{2k}(p_0(\y)+\jap{y}^{-\m}).
\end{align*}

\noindent$(ii)$
Next, we consider the case where $|\y|\leq 1/2$. Then, it suffices to prove
\begin{align*}
|\x|^{2k}\lesssim (1+\jap{\x}^2\jap{x}^{-\m'}|x-y|^2+\jap{x}^2|\x-\y|^2 )^{2k}(|\y|^{2k}+\jap{y}^{-\m})
\end{align*}
since $p_0(\x)\lesssim|\x|^{2k}$ and $|\y|^{2k}\lesssim p_0(\y)$ in this case.

\noindent$(ii-1)$
We suppose that $\jap{x}^{2k}\jap{y}^{-\m}\geq 2^{-k}$ with a constant $c>0$.
Then
\begin{align*}
|\x|^{2k}\lesssim& |\x-\y|^{2k}+|\y|^{2k}\\
\lesssim&|\x-\y|^{2k}\jap{x}^{2k}\jap{y}^{-\m}+|\y|^{2k}\\
\lesssim&(1+\jap{\x}^2\jap{x}^{-\m'}|x-y|^2+\jap{x}^2|\x-\y|^2 )^{2k}(|\y|^{2k}+\jap{y}^{-\m}).
\end{align*}

\noindent$(ii-2)$ We suppose that $\jap{x}^{2k}\jap{y}^{-\m}\leq 2^{-k}$. By Lemma \ref{xylemma}, we have $\jap{y}\lesssim |x-y|$.
Since $0<\m<2k$, we have $\jap{y}^{\m}\leq \jap{y}^{2k}\lesssim |x-y|^{2k}$. Using this inequality, we obtain
\begin{align*}
|\x|^{2k}&\lesssim  |\x-\y|^{2k}+|\y|^{2k}\underbrace{\lesssim}_{\jap{y}^{\m}\lesssim |x-y|^{2k}}\jap{y}^{-\m}|x-y|^{2k}|\x-\y|^{2k}+|\y|^{2k} \\
\underbrace{\leq}_{1\leq \jap{x},\jap{\x}}& \jap{x}^{2k-\m}\jap{\x}^{2k} \jap{y}^{-\m}|x-y|^{2k}|\x-\y|^{2k}+|\y|^{2k}\\
\underbrace{\leq}_{\m=k\m'}&(1+\jap{\x}^2\jap{x}^{-\m'}|x-y|^2+\jap{x}^2|\x-\y|^2 )^{2k}(|\y|^{2k}+\jap{y}^{-\m}),
\end{align*}
where we use the estimate
\begin{align*}
\jap{x}^{2k-\m}\jap{\x}^{2k} \jap{y}^{-\m}|x-y|^{2k}|\x-\y|^{2k}=& (\jap{x}^{2-\m'}\jap{\x}^2|x-y|^{2}|\x-\y|^{2})^k \jap{y}^{-\m}\\
\leq&((1+\jap{\x}^2\jap{x}^{-\m'}|x-y|^2+\jap{x}^2|\x-\y|^2)^2  )^k \jap{y}^{-\m}
\end{align*}
 in the last line. This completes the proof.

\end{proof}

Finally, we prove:
\begin{lem}
Let $N$ be an integer such that $N>m$. For $|\x|\geq 1$, we have
\begin{align*}
p_0(\x)\lesssim (1+\jap{\x}^2\jap{x}^{-\m'}|x-y|^2+\jap{x}^2|\x-\y|^2 )^{N}(p_0(\y)+\jap{y}^{-\m}).
\end{align*}
for $|\x|\geq 1$.
\end{lem}

\begin{proof}

Let $N$ be the smallest integer greater than or equal to $\max(k,m,\m/2)$.

\noindent$(i)$ First, we suppose $|\y|\geq \frac{1}{2}$. If $\jap{x}^{2k}\jap{y}^{-\m}\geq 2^{-k}$ holds, then
\begin{align*}
|\x|^{2m}\lesssim& |\x-\y|^{2m}+|\y|^{2m}\underbrace{\lesssim}_{\mathclap{\jap{x}^{2k}\jap{y}^{-\m}\geq 2^{-k}}} \jap{x}^{2k}\jap{y}^{-\m}|\x-\y|^{2m}+ |\y|^{2m}\\
\lesssim&(1+\jap{\x}^2\jap{x}^{-\m}|x-y|^2+\jap{x}^2|\x-\y|^2 )^{\max(k,m)}(|\y|^{2m}+\jap{y}^{-\m}).
\end{align*}
Next, we consider the case $\jap{x}^{2k}\jap{y}^{-\m}< 2^{-k}$. Due to Lemma \ref{xylemma}, $\jap{y}\lesssim |x-y|$. Then we have $1\lesssim \jap{y}^{-\m}|x-y|^{2N}$ since $2N\geq \m$.
Hence
\begin{align*}
&|\x|^{2m}\lesssim  |\x-\y|^{2N}+|\y|^{2m}\underbrace{\lesssim}_{\mathclap{1\lesssim \jap{y}^{-\m}|x-y|^{2N}}}\jap{y}^{-\m}|x-y|^{2N}|\x-\y|^{2N}+|\y|^{2m}\\
&\underbrace{\lesssim}_{\mathclap{1\leq \jap{x}, \jap{\x},\, 0<\m'<2}} \jap{x}^{(2-\m')N}\jap{\x}^{2N} \jap{y}^{-\m}|x-y|^{2N}|\x-\y|^{2N}+|\y|^{2m}\\
&\leq(1+\jap{\x}^2\jap{x}^{-\m'}|x-y|^2+\jap{x}^2|\x-\y|^2 )^{N}(|\y|^{2m}+\jap{y}^{-\m}).
\end{align*}

\noindent$(ii)$ Next, we suppose $|\y|\leq \frac{1}{2}$. In this case, $|\x-\y|\geq \frac{1}{2}|\x|$ holds.

If $\jap{x}^{2k}\jap{y}^{-\m}\geq 2^{-k}$ holds, then
\begin{align*}
|\x|^{2m}\underbrace{\lesssim}_{\mathclap{|\x|\geq 1,\, N\geq m}}& |\x|^{2N}\underbrace{\lesssim}_{|\x|\geq 1,\,|\x-\y|\geq \frac{1}{2}|\x|} |\x-\y|^{2N} \underbrace{\lesssim}_{\mathclap{\jap{x}^{2k}\jap{y}^{-\m}\geq 2^{-k},\, N\geq k}}\jap{x}^{2N}\jap{y}^{-\m}|\x-\y|^{2N}\\
\lesssim&(1+\jap{\x}^2\jap{x}^{-\m}|x-y|^2+\jap{x}^2|\x-\y|^2 )^N(|\y|^{2m}+\jap{y}^{-\m}).
\end{align*}

If $\jap{x}^{2k}\jap{y}^{-\m}< 2^{-k}$ holds, then $\jap{y}\lesssim |x-y|$  by Lemma \ref{xylemma}. Then we have $1\lesssim \jap{y}^{-\m}|x-y|^{2N}$ since $2N\geq \m$. Hence
\begin{align*}
|\x|^{2m}&\underbrace{\lesssim}_{\mathclap{|\x|\geq 1,\,|\x-\y|\geq \frac{1}{2}|\x|,\, N\geq m}}  |\x-\y|^{2N}\underbrace{\lesssim}_{\jap{y}^{\m}\lesssim |x-y|^{2N}}\jap{y}^{-\m}|x-y|^{2N}|\x-\y|^{2N} \\
\underbrace{\leq}_{\mathclap{1\leq \jap{x},\jap{\x},\,0<\m'<2}}& \jap{x}^{(2-\m')N}\jap{\x}^{2N} \jap{y}^{-\m}|x-y|^{2N}|\x-\y|^{2N}\\
\underbrace{\leq}_{\m=k\m'}&(1+\jap{\x}^2\jap{x}^{-\m'}|x-y|^2+\jap{x}^2|\x-\y|^2 )^{2k}(|\y|^{2k}+\jap{y}^{-\m}).
\end{align*}
Since $p_0(\x)\sim |\x|^{2m}$ and $p_0(\y)\sim |\y|^{2k}$, these complete the proof.

\end{proof}

Now we have proved Lemma \ref{lemweight} $(i)$.

\subsection*{Proof of Lemma \ref{lemweight} $(ii)$}

It suffices to prove that 
\begin{align}\label{homxisym1}
|\pa_{\x}^{\b}\x^{\a}|\lesssim \jap{x}^{\m'|\b|}\jap{\x}^{\frac{\max(k-m,0)}{k}|\a|-|\b|}(p+\l)^{\frac{|\a|}{2k}}
\end{align}
 for all $\a,\b\in\mathbb{Z}_{\geq 0}^n$. We may assume $\b\leq \a$.
 Since $p_0(\x)\sim |\x|^{2k}$ for $|\x|\leq 1$ and $p(x,\x)^{-1}\lesssim\jap{x}^{\m}$, we have $|\pa_{\x}^{\b}\x^{\a}|\lesssim |\x|^{|\a|-|\b|}\lesssim p(x,\x)^{\frac{|\a|-|\b|}{2k}}\lesssim \jap{x}^{\frac{\m'}{2}|\b|}p(x,\x)^{\frac{|\a|}{2k}}$ for $|\x|\leq 1$, where we recall $\m'=\m/k$. This proves \eqref{homxisym1} for $|\x|\leq 1$. On the other hand, using $p_0(\x)\sim |\x|^{2m}$ for $|\x|\geq 1$, we obtain $|\pa_{\x}^{\b}\x^{\a}|\lesssim |\x|^{|\a|-|\b|}\lesssim \jap{\x}^{-|\b|}p(x,\x)^{\frac{|\a|}{2m}}$ and $p(x,\x)^{\frac{|\a|}{2m}}\sim \jap{\x}^{\frac{k-m}{k}|\a|}p(x,\x)^{\frac{|\a|}{2k}}\lesssim \jap{\x}^{\frac{\max(k-m,0)}{k}|\a|}(p+\l)^{\frac{|\a|}{2k}}$ for $|\x|\geq 1$. These inequalities imply \eqref{homxisym1} for $|\x|\geq 1$. This completes the proof.

\end{document}